\documentclass[11pt]{amsart}

\usepackage{amsmath}
\usepackage{amsthm}
\usepackage{hyperref}
\usepackage[margin=1.0in]{geometry}
\usepackage[T1]{fontenc}

\numberwithin{equation}{section}

\newtheorem{prop}{Proposition}
\newtheorem{lemma}[prop]{Lemma}

\newtheorem{thm}[prop]{Theorem}
\newtheorem{cor}[prop]{Corollary}
\newtheorem{conj}[prop]{Conjecture}
\numberwithin{prop}{section}

\theoremstyle{definition}
\newtheorem{defn}[prop]{Definition}
\newtheorem{ex}[prop]{Example}
\newtheorem{rmk}[prop]{Remark}

\DeclareSymbolFont{script}{U}{eus}{m}{n}
\DeclareSymbolFontAlphabet{\mathscr}{script}
\DeclareMathSymbol{\Wedge}{0}{script}{"5E}
\DeclareMathAlphabet{\mathrmsl}{OT1}{cmr}{m}{sl}
\newcommand{\del}{\partial}
\newcommand{\delb}{\bar{\partial}}\newcommand{\dt}{\tfrac{\partial}{\partial t}}
\newcommand{\brs}[1]{\left| #1 \right|}

\newcommand{\gD}{\Delta}
\newcommand{\gd}{\delta}
\newcommand{\gs}{\sigma}

\newcommand{\gU}{\Upsilon}

\newcommand{\gl}{\lambda}

\newcommand{\gw}{\omega}
\newcommand{\ga}{\alpha}
\newcommand{\gb}{\beta}
\newcommand{\gL}{\Lambda}
\newcommand{\N}{\nabla}

\newcommand{\LL}{\mathcal L}
\newcommand{\til}[1]{\widetilde{#1}}

\renewcommand{\bar}[1]{\overline{#1}}

\renewcommand{\i}{\sqrt{-1}}

\renewcommand{\part}{\del}
\newcommand{\bpart}{\bar{\del}}

\newcommand{\hook}{\lrcorner}

\newcommand{\bi}{\bar{i}}
\newcommand{\bj}{\bar{j}}
\newcommand{\bk}{\bar{k}}
\newcommand{\bl}{\bar{l}}

\newcommand{\bq}{\bar{q}}

\newcommand{\bs}{\bar{s}}

\newcommand{\IP}[1]{\left<#1\right>}

\newcommand{\bga}{\bar{\alpha}}
\newcommand{\bgb}{\bar{\beta}}

\newcommand{\bmu}{\bar{\mu}}

\newcommand{\HH}{{\mathscr H}}

\DeclareMathOperator{\Rc}{Rc}

\DeclareMathOperator{\tr}{tr}

\DeclareMathOperator{\Vol}{Vol}

\DeclareMathOperator{\Ham}{Ham}

\begin{document}

\title[The nondegenerate generalized K\"ahler Calabi-Yau
problem]{The nondegenerate generalized K\"ahler Calabi-Yau
problem}

\begin{abstract} We formulate a Calabi-Yau type conjecture in generalized
K\"ahler geometry, focusing on the case of nondegenerate Poisson
structure.  After defining natural Hamiltonian deformation spaces for
generalized K\"ahler
structures generalizing the notion of K\"ahler class, we conjecture unique
solvability of Gualtieri's Calabi-Yau equation within this class.  We establish
the uniqueness, and moreover show that all such solutions are actually
hyper-K\"ahler metrics.  We furthermore establish a GIT framework for this
problem, interpreting solutions of this equation as zeros of a moment map
associated to a Hamiltonian action and finding a Kempf-Ness functional.  Lastly
we indicate the naturality of
generalized K\"ahler-Ricci flow in this setting, showing that it evolves within
the given Hamiltonian deformation class, and that the Kempf-Ness functional is
monotone, so
that
the only possible fixed points for the flow are hyper-K\"ahler metrics.  On a
hyper-K\"ahler background, we establish global existence and weak convergence
of the flow.
\end{abstract}

\date{\today}

\author{Vestislav Apostolov}
\address{D\'epartment de mathe\'ematiques\\
         Universit\'e du Qu\'ebec \`a Montr\'eal\\
         Case postale 8888, succursale centre-ville
         Mongtr\'eal (Qu\'ebec) H3C 3P8}
\email{\href{mailto:apostolov.vestislav@uqam.ca}{apostolov.vestislav@uqam.ca}}

\author{Jeffrey Streets}
\address{Rowland Hall\\
         University of California\\
         Irvine, CA 92617}
\email{\href{mailto:jstreets@uci.edu}{jstreets@uci.edu}}

\maketitle

\section{Introduction}

Let $(M^{2m}, g, J)$ be a compact K\"ahler manifold, with $\Theta \in
\wedge^{m,0} (M,J)$ a holomorphic volume form.  Yau's theorem (\cite{YauCC})
asserts
that in any
K\"ahler class there exists a unique Calabi-Yau (Ricci-flat) metric.  This
result provides a wide class of examples of Ricci flat metrics, which play a
central role in geometry and mathematical physics.  Since Yau's original proof a
number of new analytic techniques have been brought to bear on the problem.  For
instance the Aubin-Yau $J$-functional (see e.g. \cite{aubin}) yields a
variational characterization of Calabi-Yau metrics, which can be used to 
yield a variational proof of the Calabi-Yau theorem \cite{Berman}.  Also, as
shown by Cao~\cite{Cao}, in this setting the K\"ahler-Ricci flow with arbitrary
initial data exists for all time and converges to a Calabi-Yau metric.  In this
paper we generalize aspects of this picture to generalized K\"ahler
geometry.

Generalized K\"ahler structures first appeared through investigations into
supersymmetric sigma models \cite{Gates}, and were rediscovered in
a purely mathematical context in the work of Gualtieri \cite{Gualtieri-PhD} and 
Hitchin \cite{HitchinGCY}, and have recently
attracted interest in both the physics and mathematical communities as
natural generalizations of K\"ahler structures.  We will focus here entirely on
the biHermitian
description of generalized K\"ahler geometry (cf. \cite{AGG,Gates}). Thus,  a
generalized
K\"ahler manifold is a smooth manifold $M$ with a triple $(g, I, J)$ consisting
of two integrable almost-complex structures,  $I$ and $J$,  together with a
Riemannian metric $g$ which is
Hermitian with respect to both, such that the K\"ahler forms $\gw_I$ and
$\gw_J$ satisfy
\begin{align*}
 d^c_I \gw_I = H = - d^c_J \gw_J, \qquad dH = 0,
\end{align*}
where the first equation defines $H$, and $d^c_I = \i (\delb_I - \del_I)$.

A key feature of generalized K\"ahler geometry, observed by Hitchin
in \cite{HitchinPoisson}  (cf. \cite{AGG, Pontecorvo} for the $4$-dimensional
case) is that there are naturally associated Poisson structures.  In particular,
the tensor
\begin{align*}
\gs = [I,J] g^{-1}
\end{align*}
is a real Poisson structure, which is the real part of a holomorphic Poisson
structure with respect to both complex structures $I$ and $J$.  In this paper, 
we focus entirely on the case when this  holomorphic Poisson
structure is nondegenerate, in which case we will refer to the generalized
K\"ahler structure itself as nondegenerate.  These structures are also
referred to as ``type $(0,0)$'' in the language of generalized complex
structures.  In this case we define the
corresponding symplectic form
\begin{align*}
\Omega = \gs^{-1},
\end{align*}
which is the  common real part of holomorphic symplectic forms with respect to
$I$
and $J$.

The simplest example of a nondegenerate generalized K\"ahler structure comes
from hyper-K\"ahler geometry.  In particular, if $(M^{4n}, g, I, J, K)$ is
hyper-K\"ahler then $(M^{4n}, g, I, J)$ is a nondegenerate generalized K\"ahler
structure with $\Omega = -\tfrac{1}{2}\gw_K$.  Later, Joyce (\cite{AGG, 
Gualtieri-CMP, HitchinPoisson}) showed that one can appropriately deform this
example using an
$\gw_K$-Hamiltonian isotopy to produce non-K\"ahler, nondegenerate generalized
K\"ahler structures.  We rederive this construction in the purely biHermitian
context in \S \ref{s:variations}, and moreover show that the proof adapts to
show a more general statement, namely that for an arbitrary nondegenerate
generalized K\"ahler structure, $\Omega$-Hamiltonian isotopies act locally to
produce new nondegenerate generalized K\"ahler structures (Proposition
\ref{nondegvariations}).

Given this variational space for nondegenerate generalized K\"ahler
structures, it is natural to seek canonical representatives of this
class.  A Calabi-Yau equation in this setting was defined by Gualtieri
(\cite{Gualtieri-PhD} Definition 6.40).  We define this equation only
referencing biHermitian geometry in \S \ref{s:GCYeqn}, noting here that, in
analogy with the classical Calabi-Yau equation, it asks for constancy of a
certain ``Ricci potential,'' denoted $\Phi$, which is defined as the ratio of the top exterior
powers of the closed spinors defining the two relevant generalized complex
structures.   In the nondegenerate setting, this Ricci potential has a simple expression in terms of the given bihermitian triple, namely
\begin{align*}
\Phi = \log \frac{\det \left( I + J \right)}{\det \left(I - J \right)}.
\end{align*}
The relevant Calabi-Yau type equation is then
\begin{align*}
\Phi \equiv \gl,
\end{align*}
where $\gl$ is a topological invariant of the $\Omega$-Hamiltonian deformation class (cf. Lemma \ref{lambdalemma}).
Our first main result is that, within our given deformation class,
solutions to the equation are unique, and moreover more rigid than expected:
they are hyper-K\"ahler.

\begin{thm} \label{thm:uniquerigidity}  (cf. Theorem \ref{thm:uniquerigidity2})
Let $(M^{4n}, g, I, J)$ be a nondegenerate generalized K\"ahler manifold.  Any
two solutions $(g_i,I,J_i)$, $i=1,2$ of the generalized K\"ahler Calabi-Yau
equation in the $\Omega$-Hamiltonian deformation class agree, and moreover
define a hyper-K\"ahler structure.
\end{thm}

Taking inspiration from the Calabi-Yau theorem, this rigidity, as well as the
further geometric and analytic results described below, we conjecture unique
solvability of the generalized K\"ahler Calabi-Yau equation in the given
$\Omega$-Hamiltonian deformation class.
\begin{conj} \label{CYconj} Let $(M^{4n}, g, I, J)$ be a nondegenerate
generalized K\"ahler manifold.  There exists a nondegenerate generalized
K\"ahler structure $(g',I,J')$ in the $\Omega$-Hamiltonian deformation class
solving the generalized K\"ahler Calabi-Yau equation.  Moreover, this resulting
generalized K\"ahler structure is unique, and hyper-K\"ahler.
\end{conj}
\noindent Of course the uniqueness and rigidity statements have been established
in Theorem \ref{thm:uniquerigidity}.  Observe that this conjecture has the
consequence that every nondegenerate generalized K\"ahler structure is an
$\Omega$-Hamiltonian deformation of a hyper-K\"ahler structure, i.e. given by
the
Joyce construction.

To give this conjecture more context, we next show that it fits into a formal
GIT
picture.  Building on the observation that there is a natural action of
$\Omega$-Hamiltonian diffeomorphisms on generalized K\"ahler structures, and
taking inspriation from prior constructions in K\"ahler geometry, we define a
closed $1$-form on the $\Omega$-Hamiltonian deformation class (seen as a
Frech\'et space) which vanishes if and only
if the underlying GK structure is Calabi-Yau.  After taking topological
considerations into account, one can construct a primitive ${\bf F}$ for this
$1$-form, which is a natural analogue of the Aubin-Yau
$J$-functional~\cite{aubin} in
K\"ahler geometry.  We go on to define a formal symplectic structure and almost
complex
structure on each  $\Omega$-Hamiltonian deformation class.  We show (cf.
Proposition \ref{momentmap}) that the natural action of $\Omega$-Hamiltonians on
the space of generalized K\"ahler structures is itself Hamiltonian with respect
to the symplectic structure we define, and compute the moment map.  Moreover,
we show that the functional ${\bf F}$ serves as a Kempf-Ness functional, meaning
that the critical points of ${\bf F}$ are the zeroes of the momentum map.  We
note here that a different symplectic action and moment map in the
context of generalized K\"ahler geometry was recently discovered by
Boulanger~\cite{boulanger} and  Goto~\cite{Gotomoment}.

Next we provide a concrete analytic approach to Conjecture \ref{CYconj} through
the
use of the generalized K\"ahler-Ricci flow (GKRF), a natural notion of Ricci
flow adapted to the context of generalized K\"ahler geometry introduced by the
second author and Tian \cite{STGK}.  First we show that, starting with
nondegenerate initial data, solutions to the flow preserve the
$\Omega$-Hamiltonian deformation class.  We also show convexity of the $\bf
F$-functional arising
from the formal  GIT picture (cf. Proposition \ref{p:energy} for a precise
statement of part (3)).

\begin{thm} \label{flowandGIT} Let $(M^{4n}, g, I, J)$ be a nondegenerate
generalized K\"ahler manifold, and let $(g_t, I, J_t)$ denote the solution to GKRF
with this initial condition.  Then
\begin{enumerate}
\item The structure $J_t$ evolves by the one-parameter family of
$\Omega$-Hamiltonian isotopies determined by the time-dependent Ricci potential.
\item The only fixed points of the flow are solutions to the generalized
K\"ahler Calabi-Yau equation.  In particular they are hyper-K\"ahler.
\item The functional ${\bf F}$ is convex along the flow. 
\end{enumerate}
\end{thm}

Building on these natural formal properties, in line with Conjecture
\ref{CYconj} it is natural to expect
that in this setting the solution to GKRF exists for all time and converges to
a hyper-K\"ahler metric (cf. Conjecture \ref{conj:GKRF}).  With this in mind we
establish some a priori estimates which are relevant to establishing the long
time existence, and moreover lead to a definitive convergence statement assuming
certain further a priori estimates.  In particular, we begin by establishing the
key fact that the Ricci potential evolves by the pure heat equation, in
line with the corresponding behavior of K\"ahler-Ricci flow.  Based on this we
are able to obtain strong a priori estimates on the gradient of
the Ricci potential.  Combining this with previous regularity results
for GKRF, we can establish a conditional resolution of Conjecture \ref{CYconj}.

\begin{thm} \label{formalflowtheorem} (cf. Theorem \ref{formalflowtheorem2}) Let
$(M^{4n},
g_0, I, J)$ be a nondegenerate generalized K\"ahler manifold.  Let $(g_t, I,
J_t)$ denote the solution to GKRF with this initial condition.  Suppose there exists a constant $\gL > 0$ such that for all times $t$ in
the maximal interval of existence, the solution satisfies
\begin{align*}
\gL^{-1} g_0 \leq g_t \leq \gL g_0.
\end{align*}
Then the solution exists for all time and converges to a hyper-K\"ahler metric.
In particular, Conjecture~\ref{CYconj} holds true.
\end{thm}

In prior work of the second author \cite{SNDG} the global existence and weak
convergence of the flow when $n=1$ was established.  The proof exploits the
classification of complex surfaces, in particular using the existence of a
background K\"ahler metric to obtain some necessary a priori estimates.  Our
final result extends this to arbitrary $n \geq 1$.  In particular, we
establish the global existence of GKRF
under the assumption that a hyper-K\"ahler metric exists, together with some weak convergence statements.

\begin{thm} \label{flowtheorem} Let $(M^{4n}, I)$ be a compact hyper-K\"ahler
manifold.  Suppose $(g,I,J)$ is a nondegenerate generalized K\"ahler structure
on $M$.  The solution to generalized K\"ahler-Ricci flow with initial condition
$(g,I,J)$ exists on $[0,\infty)$, and satisfies
\begin{align*}
 \brs{\brs{\Phi - \gl}}_{H_1^2}^2 \leq C t^{-1}.
\end{align*}
Moreover, there exists a sequence of times $\{t_j\} \to \infty$ such that
$(\gw_I)_{t_j}$ converges in the sense of currents to a closed positive $(1,1)$ current.
\end{thm}

Note that, even
with
the assumption that a hyper-K\"ahler metric exists on our given complex
manifold,
Conjecture \ref{CYconj} does not immediately follow, as it asks for
deformability of an arbitrary GK structure to a hyper-K\"ahler one, and it is
not
a priori known if this space is \emph{connected}.  Theorem~\ref{flowtheorem} at least shows that there is a smooth deformation of the underlying K\"ahler form which limits in a weak sense to a closed positive $(1,1)$-current.  If the convergence of the flow could be improved to genuine $C^{\infty}$ convergence to a K\"ahler structure, this would verify Conjecture \ref{CYconj} for
hyper-K\"ahler backgrounds, in particular yielding the connectivity of the
nonlinear space of generalized K\"ahler structures on these manifolds, a
nonobvious fact.  We note that the weak convergence statements most likely do not represent a fundamental obstruction to strong convergence, but are rather representative of how far the analytic techniques can currently carry us.  To emphasize this point, we show that strong convergence can be verified on tori~\cite{SBIPCF}, yielding the global structure of the space of nondegenerate generalized K\"ahler structures on these manifolds.

\begin{cor} \label{c:toriconv} Let $(T^{4n}, I)$ denote a torus with a K\"ahler complex structure.  Suppose $(g, I, J)$ is a nondegenerate generalized K\"ahler structure on $T^{4n}$.  The solution to generalized K\"ahler-Ricci flow with initial condition $(g, I, J)$ exists on $[0,\infty)$ and converges to a flat, hyper-K\"ahler metric.  In particular, Conjecture~\ref{CYconj} holds true.
\end{cor}

Here is an outline of the rest of this paper.  In \S \ref{s:defns} we provide
relevant background on nondegenerate generalized K\"ahler structures, precisely
describing the deformations under consideration.  In \S \ref{s:conj} we review
the relevant Calabi-Yau type equation, give a precise setup for Conjecture
\ref{CYconj}, and prove Theorem \ref{thm:uniquerigidity}.  Next in \S
\ref{s:GIT} we establish a GIT framework for Conjecture \ref{CYconj}.  We begin
our analysis of GKRF in \S \ref{flowsec}, recalling background and proving
Theorems \ref{flowandGIT} and \ref{formalflowtheorem}.  Finally, in \S
\ref{ltesec} we develop further a priori estimates for GKRF in this setting and
prove Theorem \ref{flowtheorem} and Corollary~\ref{c:toriconv}.

\subsection*{Acknowledgements}

The first author was supported in par by an NSERC Discovery Grant and is
grateful to the Institute of Mathematics and Informatics of the Bulgarian
Academy of Sciences where a part of this project was realized.  The second
author gratefully acknowledges support from the NSF via DMS-1454854, and from an
Alfred P. Sloan Fellowship.
The second author would like to thank Marco Gualtieri for helpful discussions on
this topic.  The authors thank the referee for useful comments.

\section{Nondegenerate Generalized K\"ahler structures} \label{s:defns}

In this section we establish some fundamental properties of nondegenerate
generalized K\"ahler structures which we will use throughout the paper.  In \S
\ref{NGKdef} we recall the basic definitions, and show that they lead to the
existence of a holomorphic symplectic structure on the underlying complex
manifolds.  Next in \S \ref{HSsec} we recall some fundamental aspects of
holomorphic symplectic manifolds.  We also observe that the pluriclosed metrics
associated to a
nondegenerate generalized K\"ahler structure are symplectically tamed, and in \S
\ref{s:symptamed}
we establish some technical identities associated to such structures central to
the results to follow.  In \S \ref{s:variations} we show that
Hamiltonian diffeomorphisms associated to the real part of the holomorphic
symplectic structures can be used to produce nontrivial deformations of these
generalized K\"ahler 
structures.

\subsection{Background} \label{NGKdef}

In this subsection we recall the biHermitian formulation of generalized K\"ahler
geometry, and the basic properties of the associated Poisson structures.  We
also here record definitions of connections and their curvature relevant to the
rest
of the paper, but not \S \ref{NGKdef} directly.

\begin{defn} \label{GKdef} Given a smooth manifold $M$, we say that $(g, I, J)$
is a \emph{generalized K\"ahler structure (GK structure)} if $I$ and $J$ are
integrable complex structures, $g$ is compatible with both $I$ and $J$, and
furthermore
\begin{align*} 
 d^c_I \gw_I = H = - d^c_J \gw_J, \qquad d H = 0.
\end{align*}
Associated to this structure we define
\begin{align*}
 \gs = g^{-1} [I,J] \in \wedge^2 (TM).
\end{align*}
We will call a given GK
structure \emph{nondegenerate} if $\gs$ defines a nondegenerate pairing on each
fibre of $T^*M$, and our
focus will be on nondegenerate structures throughout this paper.
\end{defn}

In general, the tensor $\gs$ is the real part of a holomorphic $(2,0)$-bivector
with respect to both complex structures $I$ and $J$ \cite{HitchinPoisson} 
(cf. \cite{AGG, Pontecorvo} for the $4$-dimensional
case).  Observe that
when $\gs$ is nondegenerate, we can define
\begin{align} \label{Omegadef}
 \Omega = \gs^{-1} = [I,J]^{-1} g.
\end{align}
From the properties of $\gs$ it follows easily that $\Omega$ is the real part of
a holomorphic symplectic $(2,0)$-form with respect to either complex structure
$I$ and $J$, and moreover the real parts agree.  We recall some fundamental
aspects of complex manifolds admitting holomorphic symplectic $(2,0)$-forms in
\S \ref{HSsec}.  As we explain in Lemma
\ref{l:symplectic-forms}
below, the data of distinct holomorphic symplectic structures with matching real
parts determines a unique nondegenerate generalized K\"ahler structure.

Associated to generalized K\"ahler structures, and more generally Hermitian
structures, are several relevant connections.  We record these definitions here
for convenience.

\begin{defn} \label{Bismutdef} Let $(M^{2n}, g, I)$ be a complex
manifold with a Hermitian metric $g$.  The \emph{Bismut connection} is defined
by
\begin{align} \label{Bismutformula}
\IP{\N^{B}_X Y, Z} =&\ \IP{\N_X Y, Z} - \tfrac{1}{2} d^c \gw (X,Y,Z),
\end{align}
where $\nabla$ stands for the Riemannian connection of $g$.
Observe that in the context of generalized K\"ahler geometry we have two Bismut
connections defined via
\begin{align*}
\IP{\N^{B,I}_X Y, Z} =&\ \IP{\N_X Y, Z} - \tfrac{1}{2} d^c_I \gw_I (X,Y,Z) =
\IP{\N_X Y, Z} - \tfrac{1}{2} H(X,Y,Z)\\
\IP{\N^{B,J}_X Y, Z} =&\ \IP{\N_X Y, Z} - \tfrac{1}{2} d^c_J \gw_J (X,Y,Z) =
\IP{\N_X Y, Z} + \tfrac{1}{2} H(X,Y,Z).
\end{align*}
\end{defn}

\begin{defn} \label{Cherndef} Let $(M^{2n}, g, I)$ be a compact complex
manifold with  a Hermitian metric $g$.  The \emph{Chern connection} is defined
by
\begin{align} \label{Chernformula}
\IP{\N^{C}_X Y, Z} =&\ \IP{\N_X Y, Z} + \tfrac{1}{2} d^c \gw (X,IY,IZ).
\end{align}
\end{defn}

These are both Hermitian connections, which necessarily have torsion in the
non-K\"ahler case.  The
curvatures of these connections arise in our analysis, and most important are
the associated representatives of the first Chern class.

\begin{defn} \label{Riccidefs} Let $(M^{2n}, g, I)$ be a compact complex
manifold with pluriclosed metric $g$.  The \emph{Bismut Ricci curvature} and
\emph{Chern-Ricci curvature} are defined by
\begin{align*}
 \rho_B(X,Y) =&\ \tfrac{1}{2} R^B(X,Y,e_i, I e_i)\\
 \rho_C(X,Y) =&\ \tfrac{1}{2} R^C(X,Y,e_i, I e_i),
\end{align*}
where $R^B$ and $R^C$ are the curvature tensors associated to the
Bismut and Chern connections respectively, and $\{e_i\}$ is an orthonormal 
basis for $(TM, g)$.
\end{defn}

\subsection{Holomorphic symplectic manifolds} \label{HSsec}

Recall that a {\it holomorphic-symplectic} structure on a complex manifold
$(M,I)$ is defined by a holomorphic $(2,0)$-form $\Omega$ which is
non-degenerate in the sense that at each point the complex linear map $\Omega:
T^{1,0}M \to \wedge^{1,0}(T^*M)$ is non-degenerate. Elementary linear algebra
shows that the non-degeneracy condition is equivalent to ${\rm Re} (\Omega)$
being non-degenerate (and therefore symplectic) or,  equivalently, $\Omega^n
\wedge {\bar \Omega}^n \neq 0$, where $4n$ is the real dimension of
$(M,I)$.  In particular, on any holomorphic-symplectic manifold $(M,I, \Omega)$,
the form $\Omega^n$ defines a trivialization of its canonical bundle
$\wedge^{2n,0}(M,I)$.

The theory of {\it compact, K\"ahler} holomorphic-symplectic manifolds has been
extensively developed by A. Beauville and F. Bogomolov, using the Calabi-Yau
theorem~\cite{yau}. We refer to \cite{Joyce-book} for an overview and  recall
below the following well-known Beauville-Bogomolov-Yau decomposition theorem
(see e.g. \cite[Prop.~6.2.2]{Joyce-book}).

\begin{thm}\label{t:decomposition}~\cite{Beauville,bogomolov,yau} Let $(M, I)$
be a compact complex manifold which admits a K\"ahler metric and a
holomorphic-symplectic form $\Omega$. Then, in any K\"ahler class of $(M,I)$
there exists a unique Ricci-flat K\"ahler metric $g$ with respect to which
$\Omega$ is parallel.  Furthermore, up to a finite cover, $(M, I, g)$ is the
product of  irreducible simply-connected hyper-K\"ahler manifolds with a flat
(hyper-K\"ahler) even dimensional complex torus.
\end{thm}

However, there exist examples of holomorphic-symplectic
structures on non-K\"ahler manifolds, for instance the Kodaira-Thurston surface
or the higher dimensional examples of Guan \cite{Guan1,Guan2,Guan3}.  Thus the
complex
structures underlying a
nondegenerate generalized K\"ahler structure are not a priori K\"ahler, even
though the only known examples arise by deformation away from hyper-K\"ahler
structures (cf \S \ref{s:variations}), where the underlying complex structures
remain K\"ahler.  What our results suggest, and what would follow from our
main conjecture, is that these examples are the \emph{only} way to construct
nondegenerate generalized K\"ahler structures, in the sense that they are all
deformable to hyper-K\"ahler structures.

\subsection{Symplectic type generalized K\"ahler structures} \label{s:symptamed}

In this subsection we slightly generalize the discussion of nondegenerate
generalized K\"ahler structures to those of symplectic type, and isolate some
necessary properties and identities useful to what follows.  To begin we recall
the definition of a taming almost complex structure on a symplectic manifold.

\begin{defn} An almost complex structure $I$ on a symplectic manifold $(M,F)$ is
{\it tamed} by the symplectic form $F$ if $F(X, IX) > 0$ for any non-zero
tangent vector $X$.  It is easily seen that this is equivalent to the statement
that
\begin{equation}\label{tamed}
-IF = g + b,
\end{equation}
where $g$ is a positive-definite $I$-invariant Riemannian metric and $b$ is a
$2$-form of type $(2,0)+ (0, 2)$.
\end{defn}

Before specializing to generalized K\"ahler structures, we record some facts
relating taming complex structures and pluriclosed metrics.
\begin{lemma}\label{symplectic-to-pluriclosed} Let $(M, I)$ be a complex
manifold, and suppose $I$ is tamed by a  symplectic form $F$. Then the
Hermitian structure $(g, I)$  defined by \eqref{tamed} is pluriclosed, i.e.
satisfies $dd^c_I \gw_I=0$.
\end{lemma}
\begin{proof} According to \eqref{tamed},  $F = \gw_I + Ib$ is closed, so that
$d\gw_I = - d I b$, i.e.  $d^c_I \gw_I = -{\bf I} d Ib = db$ (cf. the end of the
proof of Lemma~\ref{l:symplectic-forms}). \end{proof}

The converse is also true,  if we suppose that the complex manifold $(M, I)$
satisfies the $\part_I \bpart_I$-Lemma. 
\begin{lemma}\label{pluriclosed-to-symplectic} Suppose $(M,I)$ is a complex
manifold on which the $\part_I \bpart_I$-Lemma holds at degree $(1,2)$, meaning
that  the natural map from the Bott-Chern cohomology group $H^{1,2}_{BC}(M,I)$
to the Dolbeault  cohomology group $H^{1,2}_{\bpart_I}(M,I)$ is an isomorphism.
Then, any pluriclosed Hermitian metric $g$ on $(M,I)$ is obtained from a
symplectic form $F$ which tames $I$, via \eqref{tamed}.
\end{lemma}
\begin{proof} For any pluriclosed Hermitian metric $g$ on $(M,I)$,  with
K\"ahler form $\gw_I$, $\bpart_I \gw_I$ is a $d$-closed $(1, 2)$-form which
defines a trivial class in  the Dolbeault cohomology $H^{1, 2}_{\bpart_I} 
(M,I)$ and a class in the Bott-Chern cohomology $H^{1,2}_{BC}(M,I)$. The
$\part_I \bpart_I$-Lemma implies that the class of $\bpart_I \gw_I$ in
$H^{1,2}_{BC}(M, I)$ must also be trivial, i.e.  there exists a $(0,1)$-form
$\xi$ such that $\bpart_I \gw_I = \part_I \bpart_I \xi$.  Letting
$$F: = \gw_I - \bpart_I \xi - \part_I \bar \xi, $$
we have \eqref{tamed} with $Ib = -2{\rm Re}(\bpart_I \xi)$.  Furthermore,
$$d F = \bpart_{I} F + \part_I F =  \bpart_{I} \gw_I -  \part_I \bpart_I \xi  +
\part_I \gw_I - \bpart_I \part_I \bar \xi =0.$$ \end{proof}

\begin{defn} A pluriclosed Hermitian structure $(g, I)$ on $M$, associated to a
symplectic form taming $I$ via (\ref{tamed}), will be referred to as a {\it
pluriclosed Hermitian metric of symplectic type}.
\end{defn}
We next record an important identity for the Lee form associated to a
pluriclosed Hermitian metric of symplectic type which will be central to various
calculations to follow.
\begin{lemma}\label{l:Lee-forms} Let $(M, g, I)$ be a pluriclosed Hermitian
metric
of symplectic type, $F$ a symplectic $2$-form taming $I$, and $b$ the real
$(2,0)+(0,2)$ form defined by \eqref{tamed}.  Then the Lee form $\theta_I = I
\delta \gw_I$ satisfies
\begin{equation}
 \theta_I (X)= b(\theta_I^{\sharp}, X) - \langle b, (\imath_X db)\rangle_g + 
(\delta b)(X),\ \end{equation}
 where $\delta$ is the $L^2$ adjoint of $d$.
 \end{lemma}
 \begin{proof} We shall use the following well-known expression for the
covariant derivative of the K\"ahler form of a Hermitian structure $(g, I)$ (see
e.g. \cite[Ch.~IX, Prop.~4.2]{KN} or \cite[Prop.~1]{gauduchon-connection}):
\begin{equation}\label{DF}
\begin{split}
g((\nabla_X I)(Y),  Z) &=\tfrac{1}{2}\Big((d\gw_I)(X,Y, Z) - (d\gw_I)(X,IY,
IZ)\Big) \\
&= \tfrac{1}{2} \Big(-(d^c_I \gw_I)(IX, Y, Z) +
(d^c_I\gw_I)(IX, IY, IZ)\Big)\\
                        &= \tfrac{1}{2} \Big((d^c_I \gw_I)(X, Y, IZ) +
(d^c_I\gw_I)(X, IY, Z)\Big),
\end{split}                        
\end{equation}
where $\nabla$ is the Levi-Civita connection of $g$, $\gw_I=gI$ is the
fundamental
form of $I$, and we have used that  $I$ is integrable (so that $d\gw_I$ and $d^c
\gw_I$ are of type $(1,2) + (2,1)$) to go from the second line to the  third. 

We denote by $\Lambda$ the contraction with $\gw_I$ acting on a $p$-form $\psi$
by 
 \begin{equation}\label{contraction}
\Lambda (\psi):= \gw_I \mathrel{\lrcorner} \psi =
\tfrac{1}{2}\sum_{i=1}^{2m}\psi(e_i, Ie_i, \cdot, \ldots, \cdot),
\end{equation}
where $\{e_i\}$ is any  $I$-adapted orthonormal frame and $m$ is the complex
dimension of $M$.  We then can express the Lee form as (see \cite{gauduchon})
\begin{equation}\label{Lee}
d\gw_I^{m-1} = \theta_I \wedge \gw_I^{m-1},
\end{equation} 
or, equivalently, by 
\begin{equation}\label{Lee1}
\theta_I = I (\delta \gw_I) = \Lambda (d\gw_I).
\end{equation}
Recall that under the hypothesis of Lemma~\ref{l:Lee-forms} we have $d^c_I \gw_I
= db$,  so we  compute using \eqref{DF}:
  \begin{equation*}
 \begin{split}
 \theta_I (X)&= \tfrac{1}{2}\sum_{i=1}^{2m} d \gw_I(e_i, Ie_i, X) =
\tfrac{1}{2}\sum_{i=1}^{2m}db(e_i, Ie_i, I X)
                =\tfrac{1}{2}\Big(\sum_{i=1}^{2m} (\nabla_{IX} b)(e_i, Ie_i) +
2
(\nabla_{e_i} b)(Ie_i, IX)\Big) \\
                &= -\langle b, \nabla_{IX} \gw_I \rangle_g +(\delta b)(X) +
b(\theta_I^{\sharp}, X)  =  -\langle b, (\imath_X db)\rangle_g + (\delta b)(X) +
b(\theta_I^{\sharp}, X).
                \end{split}
                \end{equation*}  \end{proof}

We next shift attention to generalized K\"ahler structures with associated
taming symplectic structures.  We first give an equivalent formulation of a wide
class of generalized K\"ahler structures encompassing the nondegenerate case.

\begin{defn}\label{d:symplectic} Given a manifold $M$, a {\it generalized
K\"ahler structure of symplectic type} on $M$ is a triple $(F, I,J)$ of a real
symplectic $2$-form $F$ and integrable complex structures $I$ and $J$, such that
$I$ and $J$ are tamed by $F$ and
\begin{equation} \label{GK-symplectic}
F(IX, Y) = - F(X, JY).
\end{equation}
Equivalently,
\begin{equation}\label{basic-1}
F I = -g + b, \ \ F J= -g - b, 
\end{equation}
where $g$ is a positive-definite symmetric tensor  and $b$ is a real $2$-form  
of type $(2,0)+(0,2)$ (with respect to both $I$ and $J$).  It follows from
direct calculations (cf. Lemma~\ref{l:symplectic-forms} below) that $(g, I, J)$
is then a
generalized K\"ahler structure with $d^c_I \gw_I = db = -d^c_J \gw_J$.  By
further direct calculations one can obtain that $I + J$ is invertible, and
\begin{align} \label{l:symplectic-unique}
F = -2 g (I+ J)^{-1}, \ \ b = \tfrac{1}{2} F (I-J) = - g (I+J)^{-1}(I-J).
\end{align}
 \end{defn}

Note that any compact generalized K\"ahler $4$-manifold is of
symplectic type, provided that $I$ and $J$ induce the same orientation and
$b_1(M)$ is even, see \cite[Prop.~4]{AGG} and \cite[Prop.~4]{HitchinPoisson}. 
We next observe that a nondegenerate generalized
K\"ahler structure is of symplectic type, and compute the relevant symplectic
forms.

\begin{lemma}\label{l:non-degenerate-symplectic} Let $(M^{4n}, g, I, J)$ be a
nondegenerate generalized K\"ahler structure.  Then it is of symplectic type in
two ways, with associated symplectic forms
\begin{align} \label{Fdefs}
F_{\pm} =&\ -2 g(I \pm J)^{-1}.
\end{align}
\end{lemma}
\begin{proof} Letting 
\begin{equation}\label{AB}
A:= (I+ J), \ \ B:= (I-J), 
\end{equation}
we can compute that
\begin{equation}\label{Omega}
\begin{split}
\gs &= -AB g^{-1}, \\
\Omega_I & = -g B^{-1}A^{-1} - \i g I B^{-1}A^{-1}, \\
\Omega_J &= -g B^{-1}A^{-1} - \i gJB^{-1}A^{-1}. 
\end{split}
\end{equation}
We introduce the $(0,2)$-tensors
\begin{equation}\label{omega-def}
F_+:= -2g A^{-1}; \  b:=  -gA^{-1}B,
\end{equation}
or, equivalently, 
\begin{equation*}
- F_+ A = 2g;  \ \ F_+ B = 2b.
\end{equation*}
It follows by the very definitions of $F_+$ and $b$, 
\begin{equation}\label{definition}
F_+ I= \tfrac{1}{2} (F_+ A   + F_+ B)= -g  + b, \ \ F_+ J =\tfrac{1}{2}(F_+ A -
F_+ B)=-g -b.
\end{equation}
Using $AB = - BA$ and that $A$ and $B$ are skew with respect to $g$, we see that
$b$ is a skew-symmetric tensor, i.e. a $2$-form. Thus \eqref{definition} implies
that $F_+$ tames both $I$ and $J$, and that $b$ is of type $(2,0)+(0,2)$ with
respect to either $I$ or $J$, i.e. \eqref{basic-1} holds true.  In order to
conclude that $(g, I, J)$ is of symplectic type it is enough to show that $F_+$
is  closed. Under the non-degeneracy assumption for $(g, I, J)$ we have 
\begin{equation*}
\begin{split}
F_+ &= -2 g A^{-1} = -2 g  A^{-1}  \gs\gs^{-1}\\
               &= 2 g A^{-1} (AB)g^{-1} \gs^{-1}  = 2B \gs^{-1}\\
               &= 2I\gs^{-1} - 2J\gs^{-1}= -2gI B^{-1}A^{-1} + 2gJ B^{-1}A^{-1}
\\
               &= 2\Big({\rm Im}(\Omega_I) - {\rm Im}(\Omega_J)\Big).
               \end{split}
               \end{equation*}  
Since ${\rm Im} (\Omega_I)$ and ${\rm Im} (\Omega_I)$ are closed it follows that
$F_+$ is closed.  The arguments above can be repeated for $F_-$, and the lemma
follows.               
\end{proof}
We now specialize to the case  of a generalized K\"ahler structure $(g, I, J)$
of symplectic type, so that both $(g, I)$ and $(g, J)$ are pluriclosed Hermitian
metrics of symplectic type,  corresponding to the same symplectic form $\omega$
and the same up to sign $2$-form $b$.
\begin{lemma}\label{l:half-non-degenerate} Let $(M^{2m}, g, I, J)$ be a
generalized
K\"ahler manifold for which ${\rm det} (I+ J)\neq 0$ and $b= -g (I+
J)^{-1}(I-J)$ be the $2$-form  of type $(2,0)+(0,2)$ with respect to either $I$
or $J$. Then,
\begin{equation}\label{step2}
\begin{split}
d\Big(\log \det (I+ J)\Big) (X) 
=\ -2\langle (\imath_X H), b \rangle_g.
\end{split}
\end{equation}
\end{lemma}
\begin{proof} Using  $d \log {\rm det} A = {\rm tr} A^{-1}
\nabla A$ and that $(I+J)$ hence also $(I+ J)^{-1}$ is skew,  we calculate
\begin{equation}\label{step1}
\begin{split}
d\Big(\log \det (I+ J)\Big) (X)  & =  {\rm tr }\Big( (I+J)^{-1} (\nabla_X I +
\nabla_X J)
\Big) 
       = - \Big\langle (\nabla_X I + \nabla_X J), (I+J)^{-1} \Big\rangle_g.
       \end{split}
\end{equation}
Recall that $H=d_I^c \gw_I= -d^c_J \gw_J$ is the torsion $3$-form of $(g, I,
J)$, so
substituting \eqref{DF} into \eqref{step1} we obtain \eqref{step2}, 
where our convention is that the inner product induced on
$\wedge^2(T^*M)$ is one half  the inner product  induced from ${\rm End}(TM)$
via the riemannian metric $g$. \end{proof}

\subsection{Variations of structure} \label{s:variations}

In this subsection we exhibit a natural class of variations of nondegenerate
generalized K\"ahler structures.  The central observation is due to Joyce, who
showed how to construct large families of generalized K\"ahler
structures in the $4$-dimensional case by deforming away from hyper-K\"ahler
structures appropriately using Hamiltonian diffeomorphisms (cf. \cite{AGG}). 
This construction was extended to arbitrary
dimensions by Gualtieri (\cite{Gualtieri-CMP} Example 2.21).  These
constructions focused on
deformation away from hyper-K\"ahler structure, and below we show that these
ideas
also yield deformations of arbitrary nondegenerate generalized K\"ahler
structures.  The first step is a higher dimensional extension of \cite[Thm.~2]{AGG},  see also  \cite[Sect. 2.2]{HitchinPoisson}.

\begin{lemma}\label{l:symplectic-forms} Suppose $(I, \Omega_I)$ and $(J,
\Omega_J)$ are two holomorphic-symplectic structures on $M^{4n}$, such that
\begin{enumerate}
\item ${\rm Re}(\Omega_I)= {\rm Re}(\Omega_J)$,
\item The $(1,1)$-part with respect to $I$ of the $2$-form $-{\rm
Im}(\Omega_J)$ is positive definite.
\end{enumerate}
Then the $I$-Hermitian metric defined by $g(X,X)=-2{\rm Im}(\Omega_J) (X, IX)$
is also $J$-invariant, and $(g, I, J)$ defines a nondegenerate generalized
K\"ahler
structure with $\Omega = {\rm Re}(\Omega_I)$.

\begin{proof} Because of condition (1), we can write the holomorphic
symplectic forms as
\begin{equation*}
\Omega_I = \Omega +\i I \Omega, \qquad \Omega_J = \Omega +\i J \Omega,
\end{equation*}
where $\Omega$ is a real symplectic form on $M$ of complex type $(2,0) + (0,2)$
with respect to both $I$ and $J$. Setting
\begin{equation*}
F := 2\big({\rm Im}(\Omega_I) - {\rm Im}(\Omega_J) \big)= 2(I\Omega - J\Omega)
\end{equation*}
we obtain another real symplectic  form which, by condition (2), tames the
complex structure $I$. Let 
\begin{equation}\label{I-decomposition}
F(X , IY)= g(X,Y) + b(X,Y)
\end{equation} 
be the decomposition of $- IF$ as the sum of a  symmetric tensor $g$ 
(which is a Riemannian metric compatible with $I$ according to (2)) and a
$2$-form $b$ (which is of complex type $(2,0)+(0,2)$ with respect to $I$).
As 
\begin{equation*}
F I = 2(I\Omega I - J \Omega I )= 2(\Omega + JI\Omega) = JF,
\end{equation*}
or, equivalently, 
\begin{equation}\label{conjugate}
F(IX, Y)= -F(X, JY),
\end{equation}
we obtain that $g$ is also $J$-invariant,  and (by using \eqref{conjugate})
\begin{equation}\label{J-decomposition}
F(X, JY) = g(X,Y) - b(X,Y).
\end{equation}
Let us denote  by $\gw_I= (F)^{1,1}_I$ and $\gw_J=(F)^{1,1}_J$ the
K\"ahler forms of $(g, I)$ and  $(g, J)$, respectively.  According to
\eqref{I-decomposition} and \eqref{J-decomposition},  we have
\begin{equation*}
F= \gw_I + Ib = \gw_J -Jb.
\end{equation*}
As $F$ is closed, 
$$d\gw_I = - dIb, \ \ d\gw_J = d Jb,$$
so that, setting ${\bf I} \eta = - \eta(I, I, I)$ for a $3$-form $\eta \in
\wedge^3(M)$,  
\begin{equation*}
d^c_I \gw_I = {\bf I} d\gw_I = -{\bf I} d Ib, \ \ d^c_J \gw_J = {\bf J} d\gw_J =
{\bf J}
d Jb.
\end{equation*}
As $b$ is of type $(2,0)+(0,2)$ with respect to $I$  we can write 
$$b= b^{2,0} + b^{0,2}, \ I b = -ib^{2,0} + ib^{0,2}.$$
As  $dIb= -d\gw_I$ is of type $(2,1) + (1,2)$, we deduce
$$d (Ib) = -i \delb b^{2,0} + i\del b^{0,2},$$
and therefore ${\bf I} d (Ib) = -\delb b^{2,0} -\del b^{0,2} = -db$. Similarly,
${\bf J} d (Jb) = -db$. We conclude, therefore, that $(g, I, J)$ is generalized
K\"ahler with 
$H= d^c_I \gw_I = db = - d^c_J \gw_J.$ 
\end{proof}
\end{lemma}

Lemma \ref{l:symplectic-forms} yields a natural construction of generalized
K\"ahler manifolds.  In fact every nondegenerate generalized K\"ahler structure
arises from this description.  We state this in the next lemma, whose proof is
contained in the proof of Lemma \ref{l:non-degenerate-symplectic}.

\begin{lemma}\label{c} Let $(M^{4n}, g, I, J)$ be a nondegenerate generalized
K\"ahler structure and $\Omega_J$ the holomorphic symplectic structure
associated to $J$. Then the $(1,1)$ part of $-2{\rm Im}(\Omega_J)$ with respect
to $I$ is positive definite and equals $\omega_I$, i.e. $(g, I, J)$ is given by
the construction of Lemma~\ref{l:symplectic-forms}.
 \end{lemma}

With this description in place we can now exhibit a natural class of
deformations of nondegenerate generalized K\"ahler structures.

\begin{prop} \label{nondegvariations} Let $(M^{4n}, g, I, J)$ be a nondegenerate
generalized K\"ahler
manifold.  Let $f_t$ be a smooth time dependent function on $M$, and $X_f$ be
the
$\Omega$-Hamiltonian vector field associated to $f$, i.e.
\begin{align*}
 df =&\ -X_f \hook \Omega.
\end{align*}
Let $\phi_t$ be the $1$-parameter family of diffeomorphisms of $M$ generated by
$X_f$.  Then for all $t$ such that the $(1,1)$-part with respect to $I$ of
$-{\rm Im} (\Omega_{\phi_t^*J})$ is positive definite, the triple $(I, \phi_t^*
J, \Omega)$ are the complex structures and symplectic structure associated to a
unique nondegenerate generalized
K\"ahler structure.
 \begin{proof} As in Lemma \ref{l:symplectic-forms}, associated to the given
generalized K\"ahler structure is a pair of complex symplectic forms written as
 \begin{equation*}
\Omega_I = \Omega +\i I \Omega, \ \Omega_J = \Omega +\i J \Omega,
\end{equation*}
where $\Omega$ is a real symplectic form on $M$ of complex type $(2,0) + (0,2)$
with respect to both $I$ and $J$.  Now define a one-parameter family of
$2$-forms via $(\Omega_J)_t = \phi_t^*(\Omega_J)$ Certainly by construction
$(\phi^* J, \phi^* \Omega_J)$ remains a holomorphic symplectic structure.  Since
$\phi$ is $\Omega$-Hamiltonian, it follows that
 \begin{align*}
  (\Omega_J)_t =&\ \phi^* (\Omega + \i J \Omega) = \Omega + \i (\phi^* J)
\Omega,
 \end{align*}
and so ${\rm Re}((\Omega_J)_t) = {\rm Re}(\Omega_I)$ for all $t$.  The result
follows from Lemma \ref{l:symplectic-forms}.
\end{proof}
\end{prop}

With this proposition in place we give a definition central to this work.

\begin{defn} \label{d:positiveHamilt} Let $(M^{4n}, g, I, J)$ be a nondegenerate
generalized K\"ahler manifold.  An $\Omega$-Hamiltonian diffeomorphism $\phi$ is
\emph{positive} if
\begin{align*}
-\left( {\rm Im}(\Omega_{\phi^* J}) \right)^{1,1}_I > 0.
\end{align*}
By Proposition \ref{nondegvariations}, a positive $\Omega$-Hamiltonian
diffeomorphism defines a unique generalized K\"ahler structure, which we denote
$(g_{\phi}, I, J_{\phi})$.  It is important to note that while $J_{\phi} =
\phi_t^* J$, the metric $g_{\phi}$ is defined implicitly from the triple $(I,
\phi_t^* J, \Omega)$, and is \emph{not} a pullback of the given $g$.
\end{defn}

With these constructions in place we can also recover the aforementioned
construction 
of Joyce of nontrivial nondegenerate generalized K\"ahler
structures by deformation away from hyper-K\"ahler structures.

\begin{cor} (cf.~\cite{AGG, Gualtieri-CMP, HitchinPoisson}) \label{p:joyce} Let
$(M^{4n}, g, I,
J, K)$ be a compact hyper-K\"ahler manifold.  For any smooth function $f$,
denote by $\phi_t$ the $\gw_K$-Hamiltonian flow determined by $f$. Then, for
$|t|$ small enough, $\phi_t$ is postive, and the associated structure $(g_t,
I, J_t)$ is not
K\"ahler,  unless $f\equiv const$.
\begin{proof} Proposition \ref{nondegvariations} already yields that the
deformation generates generalized K\"aher structures, so we need to show
that $(g_t, I)$ is not  K\"ahler  for $|t|$ small enough,
when the generating function $f$ is not constant.  To this end, consider the
angle
function defined by
\begin{equation}\label{p(t)}
p_t= -\tfrac{1}{4n}{\rm tr}(IJ_t)= \tfrac{1}{4n}\langle I, J_t \rangle_{g_t}.
\end{equation}
Clearly, if $(g_t, I)$ were K\"ahler, so would be $(g_t, J_t)$, and hence $p_t$
would be a constant function on $M$ as $I$ and $J_t$ are $g_t$-parallel.  One
notes that with the given hyper-K\"ahler backgorund, $X_f = \tfrac{1}{2} [I, J]
g^{-1}(df)$. Hence differentiating \eqref{p(t)}, we obtain
\begin{equation*}
\begin{split}
\tfrac{dp_t}{dt}\Big|_{t=0} &= \tfrac{1}{4n} {\rm tr} \Big(I  \circ ({\mathcal
L}_{X_f}
J)\Big) =  -\tfrac{1}{4n} {\rm tr} \Big(I \circ
[\nabla X_f, J]\Big)
                                       =  \tfrac{1}{8n} {\rm tr} \Big(
[I,J]\circ [I,J] \circ \nabla g^{-1}(df) \Big) \\
                                       &= - \tfrac{1}{2n}\Delta_g f,
\end{split}
\end{equation*}                                       
where $\nabla$ is the Levi-Civita connection of the hyper-K\"ahler metric $g$,
and $\gD_g = - d\delta^g - \delta^g d$ will denote the `analytic' Laplace
operator in accordance with the notation of \cite{STGK}.  This maximum principle
shows that if
$f$ is not a constant function, $p_t$ cannot stay constant for $t$ sufficiently
small.
\end{proof}
\end{cor}

\section{A Calabi-Yau conjecture for nondegenerate generalized K\"ahler
manifolds} \label{s:conj}
\subsection{Generalized Calabi-Yau structures} 

In this subsection we recall some motivating ideas from the theory of
generalized complex geometry which lead to the definition of the generalized
Calabi-Yau equation.  Our exposition is brief, and we refer to
\cite{Gualtieri-PhD, Gualtieri-Ann, HitchinGCY} for the general theory of
generalized complex structures.
The notion of an (even) generalized Calabi-Yau structure was introduced by
Hitchin~\cite{HitchinGCY} as  a special example of a generalized complex
structure,
i.e. an almost complex structure $\mathbb I$ defined on the vector bundle $TM
\oplus
T^*M$, which is orthogonal with respect to the natural non-degenerate inner
product 
\begin{equation}\label{product}
\langle X+ \xi, X + \xi \rangle : = -\xi(X), \ X \in TM, \ \xi \in T^*M, 
\end{equation}
and is integrable in the sense that the $i$-eigenspace $E \subset (TM \oplus
T^*M )\otimes {\mathbb C}$ of $\mathbb I$ 
is closed under the Courant bracket on $TM \oplus T^*M$. 

The case of (even) generalized Calabi-Yau structures studied in
\cite{HitchinGCY}
corresponds to the special situation when the $i$-eigenspace of the generalized
complex structure ${\mathbb I}$ coincides with the annihilator $E_{\varphi}$ of
an
even-degree {\it closed} form $\varphi \in \Gamma(\wedge^{\rm ev} (M)\otimes
{\mathbb C})$,  under the natural pointwise action of $TM\oplus T^*M$ on the
exterior algebra
$\wedge^{\bullet} (M)$  of $T^* M$,
$$(X, \xi) \cdot \varphi  := \imath_X \varphi + \xi\wedge \varphi.$$
As  $(X, \xi)^2 \cdot \varphi = -\langle X+\xi, X+\xi \rangle \varphi$, 
the annihilator $E_{\varphi}$ of a non-zero even form  $\varphi$ is
automatically
isotropic with respect to the $\mathbb C$-linear extension of the product
$\langle \cdot, \cdot \rangle$ on $(TM\oplus T^*M)\otimes {\mathbb C}$.  It is
shown in \cite{HitchinGCY} that $d\varphi=0$ implies that $E_{\varphi}$ is
closed
under
the Courant bracket. Furthermore, the condition  $E_{\varphi}\cap \bar
E_{\varphi}=0$
can be equivalently expressed in terms of the natural $\wedge^{2m}(M)$-valued
inner-product  on $\wedge^{\rm ev} (M )\otimes \mathbb C$, defined by
$$\langle \varphi, \psi \rangle = \sum_{r} (-1)^r \varphi_{2r}\wedge
\psi_{2m-2r}, $$
where $2m$ denotes the real dimension of $M$, and we write $\varphi = \sum_r
\varphi_{2r}$ for the degree decomposition of a form in $\wedge^{\rm ev} (M)$.
It
turns out that  for any $\varphi, \psi \in \wedge^{\rm ev}(M)$, 
$E_{\varphi}\cap
E_{\psi} \neq 0$ iff $\langle \varphi, \psi \rangle =0$. Thus,  fixing an
orientation of $M$,  $\varphi$ must satisfy (at every point)
\begin{equation}\label{pure-spinor}
\langle \varphi, \varphi \rangle =0, \ \ \ \langle \varphi, \bar \varphi \rangle
>0.
\end{equation}
Finally, the fact that $E_{\varphi}$ has maximal dimension ($={\rm dim}_{\mathbb
R}(M)=2m$) leads to a complicated non-linear pointwise algebraic condition on 
$\varphi$, which is referred in the literature to as $\varphi$ being a  {\it
pure
spinor}. To summarize, 
\begin{defn}\label{d:hitchin} An (even) generalized Calabi-Yau structure on $M$
is defined by a closed pure spinor $\varphi \in \Gamma(\wedge^{\rm ev} (M))$
satisfying \eqref{pure-spinor} at each point of $M$.
\end{defn}

The main observation in \cite{HitchinGCY}, and the reason for the terminology,
is
the following prototypical example, showing that the complex and K\"ahler
structures underlying a Calabi-Yau manifold can both be interpreted as even
generalized Calabi-Yau structures.

\begin{ex}  If  $\Theta$ is a holomorphic trivialization of the canonical bundle
$K(M,I) = \wedge^{m,0}(M, I)$ of a complex manifold $(M^{2m}, I)$, then $\Theta$
viewed as a closed  complex $(m,0)$  form on $M$, defines a closed pure spinor
giving
rise to the generalized  complex structure
$${\mathbb I}_{\Theta} = \left(\begin{array}{cc}- I & 0 \\ 0 & -I
\end{array}\right).$$
Similarly,  if $\omega$ is a real symplectic $2$-form on $M$, then  
$$\varphi=\exp{\i \omega}= 1 + \i \omega -\tfrac{1}{2}\omega^2 + \cdots$$
is  a closed pure spinor, and the corresponding generalized complex structure
${\mathbb I}_{\varphi}$ takes the form (see \cite[Ex. 4.6]{Gualtieri-PhD})
\begin{equation*}
{\mathbb I}_{\omega} = \left(\begin{array}{cc} 0 & - \omega^{-1} \\ \omega & 0
\end{array}\right).
\end{equation*}
In particular,  any  Calabi-Yau manifold, i.e. a compact complex manifold
$(M,I)$ which admits both a K\"ahler metric $\omega$ and a trivialization
$\Theta$ of its canonical line bundle,  is naturally endowed with two
generalized Calabi-Yau structures, given by the closed pure spinors
$\varphi_1=\exp{\i\omega}$ and $\varphi_2=\Theta$.
\end{ex}

To make the link with the non-degenerate generalized K\"ahler structures studied
in this paper, we recall that Gualtieri~\cite{Gualtieri-PhD, Gualtieri-CMP}
has shown that the data $(g, I, J, b)$ of a Riemannian metric $g$, two
$g$-orthogonal  integrable almost-complex structures $(I, J)$, and a $2$-form
$b$ on $M$,  satisfying the relation
$$d^c \gw_I = db = -d^c_J \gw_J$$
determine and are determined by a pair of {\it commuting} generalized complex
structures ${\mathbb I}_{\pm}$ on $TM \oplus T^*M,$ such that ${\mathbb I}_+
\circ {\mathbb I}_-$ is positive-definite with respect to \eqref{product}.
Specifically, given $(g, I, J, b)$ as above, the generalized complex structures
$\mathbb I_{\pm}$ are
\begin{equation}\label{GK-true}
\mathbb I_{\pm} = \tfrac{1}{2} e^b \left(\begin{array}{cc} (I \mp J) &
-(\gw_I^{-1} \pm \gw_J^{-1}) \\ \gw_I \pm \gw_J & (I \mp J )\end{array}\right)
e^{-b},
\end{equation}
where $e^b :=\left(\begin{array}{cc}  1 & 0 \\ b & 1\end{array}\right)$ is the
exponential  of  $\left(\begin{array}{cc}  0 & 0 \\ b & 0\end{array}\right)$
(and we note that our convention for the action of $I, J$ on $T^*M$ defers by a
sign with the one used in \cite{Gualtieri-PhD}). By
Lemma~\ref{l:non-degenerate-symplectic}, for any non-degenerate GK structure
$(g, I,  J)$ in the sense of the present paper (see Definition~\ref{GKdef}),  
we
can take  $b= g (I-J) (I+J)^{-1}$. A straightforward computation shows that with
this choice of $b$, \eqref{GK-true} becomes
 \begin{equation}\label{GK-symplectic-true}
\mathbb I_{-} = e^{-4\Omega} \left(\begin{array}{lc} 0 & - (F_-)^{-1} \\ F_- & 0
\end{array}\right)e^{4\Omega}, \ \ \mathbb I_{+}= \left(\begin{array}{lc} 0 & -
(F_+)^{-1} \\ F_+ & 0 \end{array}\right),
\end{equation}
where $F_{\pm} = -2g(I \pm J)^{-1}$ are the  symplectic forms taming $(I, \pm
J)$ via Lemma~\ref{l:non-degenerate-symplectic} and
$\Omega=g(I+J)^{-1}(I-J)^{-1}$ is the common real part of the
holomorphic-symplectic forms $\Omega_I$ and $\Omega_J$.    In particular,
$\mathbb I_{\pm}$  are generalized
Calabi-Yau structures corresponding to the closed pure spinors $\varphi_- =
e^{4\Omega +\i F_-}$ and $\varphi_+= e^{\i F_+}$. The above formula
extends to the non-degenerate generalized K\"ahler case the observations made
in \cite[Examples 6.30 \& 6.31]{Gualtieri-PhD}  in the hyper-K\"ahler  and the
nondegenerate
$4$-dimensional cases.

\subsection{Generalized Calabi-Yau equation} \label{s:GCYeqn}

In this subsection we recall the definition of the generalized K\"ahler
Calabi-Yau equation, and describe it in terms of biHermitian data in our
setting.  To motivate the definition, recall that a classical Calabi-Yau
K\"ahler metric on $(M^{2m}, I)$, with holomorphic volume form $\Theta$, is
defined by an $I$-compatible symplectic $2$-form $\omega$ such that $\Theta
\wedge \bar \Theta = \gl \omega^{m},$ for some constant $\gl$.  We note that
this condition is equivalently expressed in terms of the corresponding pure
spinors $\varphi_-=\exp{\i\omega}, \varphi_+=\Theta$ as
\begin{equation}\label{generalized-calabi-yau-metric}
\langle \varphi_+, \bar \varphi_+ \rangle = \gl \langle \varphi_-, \bar
\varphi_- \rangle.
\end{equation}

\begin{defn}  Let $(M^{2m}, \mathbb I_+, \mathbb
I_-)$ be a generalized K\"ahler manifold where $\mathbb I_{\pm}$ are even
generalized Calabi-Yau structures defined by pure spinors $\varphi_{\pm}$. 
Define
the \emph{Ricci potential} $\Phi$ via
\begin{align*}
\Phi = \log \frac{(\varphi_+, \bar{\varphi}_+)}{(\varphi_-, \bar{\varphi}_-)}.
\end{align*}
We say that the structure is \emph{generalized K\"ahler Calabi-Yau} if
\begin{align*}
  \Phi \equiv \gl
\end{align*}
for some $\gl \in \mathbb R$.
\end{defn}

The terminology ``Ricci potential'' certainly matches with the situation in the
K\"ahler case, where we see that $- \i \del \delb \Phi = \rho$, the usual Ricci
form.  Moreover, in that setting the equation $\Phi \equiv \gl$ certainly
defines a Ricci-flat, Calabi-Yau metric.  We will justify this terminology in
our setting in Proposition \ref{ricci-potential}.  We note that the geometry of
pairs of even generalized Calabi-Yau structures satisfying
\eqref{generalized-calabi-yau-metric} on $K3$ complex surface has been studied
by D. Huybrechts in \cite{huybrechts}.  A detailed study of the local
description of solutions to this equation appeared in \cite{HullGCY}.   In the
next lemma we reduce the
generalized Calabi-Yau equation to one involving purely biHermitian data in the
nondegenerate case.

\begin{lemma}  \label{GKCYspinors} Let $(M^{4n}, g,
I, J)$ be a nondegenerate generalized
K\"ahler manifold.  Then
\begin{align*}
 \Phi = \log \tfrac{F_+^{2n}}{F_-^{2n}}.
\end{align*}
\begin{proof}  Returning to the notation of the
previous subsection, it follows from \eqref{GK-symplectic-true} that (see
\cite[Example~4.10]{Gualtieri-PhD}) both $\mathbb I_{\pm}$  are generalized
Calabi-Yau structures corresponding to the closed pure spinors $\varphi_- =
e^{4\Omega +\i F_-}$ and $\varphi_+= e^{\i F_+}$. Furthermore,  as shown for
instance
in \cite[pp. 286-287]{HitchinGCY}, we have 
$$\langle e^{4\Omega + \i F_-}, e^{4\Phi -i F_-}\rangle = \langle e^{\i F_-},
e^{-\i
F_-}\rangle= \tfrac{2^{2n}}{(2n)!} (F_-)^{2n}, \ \ 
\langle e^{\i F_+}, e^{-\i F_+}\rangle = \tfrac{2^{2n}}{(2n)!} (F_+)^{2n},$$
finishing the lemma. \end{proof}
\end{lemma}

\subsection{A generalized Calabi conjecture}\label{s:GCY}

Let $(M^{4n}, g , I, J)$ be a compact nondegenerate generalized K\"ahler
structure, with $\Omega = g [I,J]^{-1}$ as above.
Proposition~\ref{nondegvariations} shows that there exists a natural local
action of the group ${\rm Ham}(M, \Omega)$ of Hamiltonian isotopies with respect
to $\Omega$ on the space of nondegenerate generalized K\"ahler structures, also
fixing $I$.  
Before stating our conjecture, we record a basic lemma indicating an integral
invariant associated to this action, which determines the possible value of
$\gl$ in a generalized K\"ahler Calabi-Yau structure.

\begin{lemma} \label{lambdalemma} Let $(M^{4n}, g, I, J)$ be a nondegenerate
generalized K\"ahler
structure.  The quantity
\begin{equation}\label{GCY-constant}
\gl := \log \left\{ \left(\int_M F_+^{2n} \right) \Big/ \left( \int_M F_-^{2n}
\right) \right\}
\end{equation}
is an invariant of the $\Omega$-Hamiltonian deformation class.
\begin{proof} First, it is clear that the $\Omega$-Hamiltonian action preserves
$I$ and $\Omega$, and so preserves  $\Omega_I$ and the deRham cohomology class
of $\Omega_{J_t}$.  Since
\begin{align*}
F_{\pm} = 2\big({\rm Im}(\pm \Omega_I) -  {\rm Im} (\Omega_J)\big),
\end{align*}
the lemma follows.
\end{proof}
\end{lemma}

We now state our main conjecture:

\begin{conj}\label{GCY-conj} Let $(M^{4n}, g, I, J)$ be a compact nondegenerate
generalized K\"ahler manifold.  
\begin{enumerate}
\item \textbf{Existence:} There exists a positive diffeomorphism $\phi \in
{\rm Ham}(M, \Omega)$ such that 
\begin{align*}
\Phi(g_{\phi},I,J_{\phi}) \equiv \gl,
\end{align*}
where $\lambda$ is the constant introduced in \eqref{GCY-constant}.
\item \textbf{Uniqueness:} The induced generalized K\"ahler structure
$(g_{\phi}, I, J_{\phi})$ is unique.
\item \textbf{Rigidity:} The induced generalized K\"ahler structure $(g_{\phi},
I, J_{\phi})$ is hyper-K\"ahler.
\end{enumerate}
\end{conj}

In particular the conjecture claims there is a generalized K\"ahler Calabi-Yau
structure in every $\Omega$-Hamiltonian deformation class, which is unique, and
which moreover is highly rigid: it is hyper-K\"ahler.  Thus,
Conjecture~\ref{GCY-conj} would imply that {\it any} non-degenerate generalized
K\"ahler structure on a compact hyper-K\"ahler manifold $M^{4n}$ must arise from
the construction of Proposition~\ref{p:joyce}.  We verify the last two parts of
the conjecture below.

\begin{rmk}\label{goto} Recently, R. Goto~\cite{Gotomoment} introduced a notion
of a {\it generalized scalar curvature} of a generalized K\"ahler manifold of
symplectic type. In the nondegenerate case, it is given by (see \cite[p.
4]{Gotomoment} and \eqref{GK-symplectic-true}) 
\begin{equation*}
\begin{split}
{\rm Gscal}_g &= 4\Big(\tfrac{d (\Omega F_-^{-1} (d\Phi)) \wedge
F_+^{2n-1}}{F_+^{2n}}\Big) = - \tfrac{1}{4n}\Big(\Delta_g \Phi  + 
d\Phi\big(\big(\nabla_{e_i} (I+J)^{-1}\big)((I+J)(e_i))\big)\Big).
\end{split}
\end{equation*}
It follows by the maximum principle that on a compact nondegenerate generalized
K\"ahler manifold $(M^{4n}, g, I, J),$  the generalized scalar curvature ${\rm
Gscal}_g$ is constant if and only if $\Phi$ is constant, thus giving a yet
another motivation for Conjecture~\ref{GCY-conj}.                        
\end{rmk}

\subsection{Uniqueness and rigidity}

In this subsection we establish both the uniqueness and rigidity claims of
Conjecture \ref{GCY-conj}.  To begin we establish a key curvature identity
relating the Bismut-Ricci curvatures of a nondegenerate generalized K\"ahler
structure to the Ricci potential $\Phi$.  We recall that on complex surfaces, or
more
generally, when $I, J$ are compatible with an almost-quaternion structure on
$M^{4n}$ (see \cite{AMP}), $I$ and $J$ verify the identity
\begin{equation}\label{angle-qc} I J + J I = - 2p {\rm Id},  
\end{equation}  where the smooth function  $p:= -\tfrac{1}{4n} {\rm tr} IJ$ is
the so-called the {\it angle function} of $(g, I, J)$.  It follows from
\eqref{angle-qc} that $(I\pm J)^{2} = -2(1\pm p) {\rm Id},$  so we obtain
$\Phi=\log \Big(\tfrac{1-p}{1+p}\Big).$  This shows in particular that $\Phi$ is
constant when $I, J$ belong to the same hyper-K\"ahler structure on $M^{4n}$.
More generally we will show that for a nondegenerate generalized K\"ahler
structure, $\Phi$ determines all of the relevant curvature and torsion
quantities.  We begin with a lemma on the differential of the Ricci
potential.

\begin{lemma}\label{l:extension} For any non-degenerate generalized K\"ahler
structure $(g, I, J)$, one has
$$[I,J] (d\Phi) = -d\Phi \circ [I, J]  = 2(\theta_I - \theta_J).$$
\end{lemma} 

\begin{proof} By Lemmas \ref{l:Lee-forms} and \ref{l:half-non-degenerate} (and
exchanging the roles of $I$ and $J$), we have
\begin{equation*}\label{1}
\begin{split}
\theta_I = &  (I-J)(I+J)^{-1} \theta_I + \tfrac{1}{2} d \log \det (I+J) +
\delta^g (I-J)(I+J)^{-1}\\
\theta_J = & - (I-J)(I+J)^{-1}\theta_J  + \tfrac{1}{2} d \log \det (I+J) -
\delta^g (I-J)(I+J)^{-1}.
\end{split}
\end{equation*}
Summing the above two equalities and observing the identities 
\begin{equation*}
\begin{split}
(I+J)^{-1}= & + [I,J]^{-1}(I-J)= -(I-J)[I,J]^{-1}, \\
 (I-J)^{-1}= &-[I,J]^{-1}(I+J)= +(I+J) [I, J]^{-1},
 \end{split}
\end{equation*}
we obtain
\begin{equation}\label{2}
\theta_I + \theta_J = -(I-J)^2[I,J]^{-1}(\theta_I - \theta_J) + d \log {\rm det}
(I+J).
\end{equation}
Rewriting \eqref{2} with respect to $(g, I, -J)$ we also have
\begin{equation}\label{3}
\theta_I + \theta_J = (I+J)^{2}[I,J]^{-1} (\theta_I - \theta_J) + d \log \det
(I-J).
\end{equation}
Then,  \eqref{3}-\eqref{2} gives
\begin{equation}\label{4}
2d \Phi = -\Big((I-J)^2 + (I+J)^2\Big)[I, J]^{-1}(\theta_I - \theta_J)=
4[I,J]^{-1}(\theta_I - \theta_J),
\end{equation}
as required.
\end{proof}

\begin{prop}\label{ricci-potential} Let $(M^{4n}, g, I, J)$ be a non-degenerate
generalized K\"ahler manifold. Then the Bismut-Ricci forms  associated to $(g,
I)$ and $(g, J)$  are respectively given by
$$(\rho_B)_I = -\tfrac{1}{2} d J d \Phi, \ \ (\rho_B)_J = -\tfrac{1}{2} d I d
\Phi.$$
In particular, on a compact non-degenerate generalized-K\"ahler manifold
$(M^{4n}, g, I, J)$, $(\rho_B)_I=0$ iff $(\rho_B)_J =0$ iff $\Phi=\gl.$
\end{prop}
\begin{proof} Using the formulas (\ref{Bismutformula}) and (\ref{Chernformula})
one can check that the
induced unitary connections $\nabla^B$ and $\nabla^C$ on $(K^{-1}(M,I), g)$
satisfy (see also \cite[Rem.~5]{gauduchon-connection})
\begin{equation}\label{paul}
\nabla^B_X = \nabla^C_X  - i(I\theta_I) (X).
\end{equation}
It
follows
from \eqref{paul} that the corresponding Ricci curvatures are related by
\begin{equation}\label{ricci-relation}
(\rho_B)_I = (\rho_C)_I + d I \theta_I.
\end{equation}
Using that $\Omega_I^n$ is a holomorphic section of $K(M, I)$, the Chern-Ricci
form is given by  
\begin{equation}\label{chern}
\begin{split}
(\rho_C)_I  & =  \tfrac{1}{2} d I d \log\Big(\tfrac{\Omega_I^n \wedge
\bar{\Omega}_I^n}{\gw_I^{2n}}\Big) = \tfrac{1}{2} d I d
\log \Big(\tfrac{\Omega^{2n}}{\gw_I^{2n}}\Big) =  \tfrac{1}{4} d I d \log \det
[I,J]^{-1}\\
             & = - \tfrac{1}{4} d I \Big(d \log \det (I+J) + d \log  \det
(I-J)\Big) \\
             &=  \tfrac{1}{2} d I \Big(-(\theta_I + \theta_J) + (IJ +
JI)[I,J]^{-1}(\theta_I - \theta_J)\Big),
             \end{split}
             \end{equation}
where for the last equality we have used \eqref{3} and \eqref{4}. Substituting
back in \eqref{ricci-relation} and using Lemma~\ref{l:extension}, we obtain
\begin{equation}
\begin{split}
(\rho_B)_I &= \tfrac{1}{2} d I \Big((\theta_I - \theta_J)  + (IJ + JI)[I,J]^{-1}
(\theta_I - \theta_J) \Big)\\
                 &=  \tfrac{1}{2} d I \Big((IJ-JI)[I,J]^{-1} (\theta_I -
\theta_J)  + (IJ + JI)[I,J]^{-1} (\theta_I - \theta_J) \Big)\\
                 &= d I \Big( IJ [I,J]^{-1}(\theta_I - \theta_J) \Big)\\
                 &=-\tfrac{1}{2}dJ d\Phi,
                \end{split}
                \end{equation}
 as required.
 \end{proof}

\begin{prop}\label{p:CY} Let $(M^{4n}, g, I, J)$ be a compact nondegenerate
generalized K\"ahler manifold.  Then,   $(\rho_B)^{1,1}_I =0$ iff
$(\rho_B)^{1,1}_J = 0$  iff $(g, I)$ and $(g, J)$ are K\"ahler Ricci-flat
metrics belonging to the same hyper-K\"ahler structure on $M^{4n}$. In
particular, any nondegenerate generalized K\"ahler Calabi-Yau structure on
$M^{4n}$ is hyper-K\"ahler in the usual sense. 
\end{prop} 
\begin{proof}  This follows essentially from the arguments in
\cite{Alexandrov-Ivanov} (see also \cite{IP}, Theorem~4.1).  For convenience of
the reader, we reproduce them below.  We will work only with the 
Hermitian structure $(g,I)$, and to simplify notation we drop the index $I$.  We
thus need to show
that for any compact pluriclosed Hermitian manifold $(M, g, I)$ with trivial
canonical bundle satisfying 
\begin{equation}\label{weakly-CY}
(\rho_B)^{1,1} = 0,
\end{equation}
must be K\"ahler.  We are going to establish this under the weaker assumption,
namely supposing that the scalar curvature  $s_B:=2\langle\rho_B, \omega
\rangle_g$ of $\nabla^B$  is identically zero, i.e.
\begin{equation}\label{sB-integral}
 s_B =0.
\end{equation}
To this end, we are going to use the following relation between the scalar
curvatures $s_C=2\langle\rho_C, \omega \rangle_g$  and $s_B$ (which follows by 
taking a trace in \eqref{ricci-relation} with respect to $\omega$)
\begin{equation}\label{scalar-relation}
\begin{split}
s_B &= s_C + \langle d I \theta, \omega \rangle_g \\
        &= s_C  - \delta \theta - |\theta|^2_g.
  \end{split}
  \end{equation}      
A key observation (\cite{Alexandrov-Ivanov}, Eq. (2. 13)) is that  the
pluriclosedness of $(g, I)$ implies the identity
\begin{equation}\label{bogdan-stefan-2}
|d\omega|_g^2 = \delta \theta + |\theta|_g^2,
\end{equation}
so that \eqref{scalar-relation} reduces to
$$s_C = s_B +  |d\omega|_g^2  = |d\omega|^2_g \ge 0$$
under the assumption \eqref{sB-integral}.  Since, furthermore, the canonical
bundle $K(M,I)$ is trivial, it follows by (\cite{gauduchon-plurigenera},
Proposition 13 and Th\'eor\`eme de Classification)  that $s_C \equiv 0$, i.e. 
$d\omega=0$,  meaning that $(g, I)$ is K\"ahler.  As on a K\"ahler manifold the
Bismut connection coincides with
the Levi-Civita connection,  $\rho_B$ is therefore the usual Ricci form and
\eqref{weakly-CY} means that $(g,I)$ is a Calabi-Yau metric.  Similarly $(g,
J)$ is a Calabi-Yau K\"ahler metric, so that $(g, I, J)$ gives rise to a
hyper-K\"ahler metric (cf. Theorem~\ref{t:decomposition}). \end{proof}

\begin{prop}\label{p:uniqueness} Suppose $\phi_1, \phi_2 \in {\rm Ham}(M,
\Omega)$  satisfy the hypotheses of Conjecture~\ref{GCY-conj}. Then, the
corresponding generalized K\"ahler structures $(g_{{\phi}_1}, I,
J_{{\phi}_1})$ and $(g_{{\phi}_2}, I, J_{{\phi}_2})$ coincide, i.e.
$g_{{\phi}_1}=g_{{\phi}_2},  J_{{\phi}_1}=J_{{\phi}_2}$. 
\end{prop}
\begin{proof} 
By assumption, each $\phi_i$ belongs to the connected component of the identity
of the group of diffeomorphisms on $M$.  It follows that each $\phi_i$ acts
trivially on deRham cohomology, so that the symplectic $2$-forms
$F_{{\phi}_1} := -2\Big({\rm Im}(\Omega_{J_{{\phi}_1}}) - {\rm
Im}(\Omega_I)\Big)$ and $F_{{\phi}_2}: = -2\Big({\rm
Im}(\Omega_{J_{{\phi}_2}}) - {\rm Im}(\Omega_I)\Big)$ taming $I$ belong to
the same deRham class (see Lemma~\ref{l:non-degenerate-symplectic}). 
Furthermore,
as $(g_{{\phi}_1}, I)$ and $(g_{{\phi}_2}, I)$ are both K\"ahler,  the
$(1,1)$-parts of $F_{{\phi}_1}$ and $F_{{\phi}_2}$  with respect
to $I$ are the K\"ahler forms $\gw_{{\phi}_1}$ and $\gw_{{\phi}_2}$ of
$(g_{{\phi}_1},I)$ and $(g_{{\phi}_2}, I)$, respectively, which thus
belong to the same Aeppli cohomology class.  Using the K\"ahler condition again,
it then follows that $\gw_{\phi_1}$ and $\gw_{\phi_2}$ define Calabi-Yau
metrics on $(M, I)$ in the same K\"ahler 
class. By the uniqueness of the Calabi-Yau metric in its K\"ahler class, we
conclude $\gw_{{\phi}_1}= \gw_{{\phi}_2}$ or, equivalently,
$g_{{\phi}_1} =
g_{{\phi}_2}$. 

Notice  that for any GK structure $(g, I, J)$  of symplectic type with
symplectic $2$-form $F$ taming $I$ and $J$ as in \eqref{GK-symplectic}, we
have (see \eqref{basic-1}) $F (I+J) = -2g$.  It thus follows that in order
to show $J_{{\phi}_1}= J_{{\phi}_2}$ it is enough to establish that
$F_{{\phi}_1}= F_{{\phi}_2}$.  As we have already proved that
the $(1,1)$-parts with respect to $I$ of $F_{{\phi}_i}$ coincide and are
closed, it follows that $(2,0)+(0,2)$-part of
$F_{{\phi}_1}-F_{{\phi}_2}$ is closed and exact.  It must
therefore be zero, being also parallel with respect to
$g_{{\phi}_1}=g_{{\phi}_2}$ (this follows from the fact that
$(g_{{\phi}_i}, I, J_{{\phi}_i})$ is hyper-K\"ahler and
(\ref{Fdefs}). We thus conclude
$F_{{\phi}_1}=F_{{\phi}_2}$ which in turn implies
$J_{{\phi}_1}=J_{{\phi}_2}$. \end{proof}

\begin{rmk}\label{r:uniqueness}  By the very definition of $J_{\phi_i}$, the
identity $J_{\phi_1}= J_{\phi_2}$ in Proposition~\ref{p:uniqueness} is
equivalent to  $\phi_1 \circ \phi_2^{-1} \in {\rm Aut}(M,J)$.  If ${\rm
Aut}(M, J)=\{{\rm id}\}$,  we can conclude $\phi_1=\phi_2$. In general,
for a holomorphic-symplectic K\"ahler manifold $(M,J)$ whose first Betti number
is zero (see Theorem~\ref{t:decomposition}) the group ${\rm Aut}(M,J)$ is
discrete.  This follows because,  on the one hand,  the connected component of 
the
identity  ${\rm Aut}_0(M,J)$ coincides with the group of reduced automorphisms
${\rm Aut}^r(M, J)$, and, on  the other hand,  ${\rm Aut}^r(M, J)$ is trivial on
a Calabi-Yau manifold by Matsushima's theorem (see e.g. \cite{gauduchon-book}).
Furthermore, $\phi_1 \circ \phi_2^{-1}$ acts trivially on $H^2(M, {\mathbb
R})$ (being in ${\rm Diff}_0(M)$). This implies that  it must belong to the
isometry group of any Calabi-Yau metric on $(M, J)$, showing that $\phi_1
\circ \phi_2^{-1}$ is also of finite order. It is well-known (see e.g.
\cite[Ch. 15, Cor. 2.6]{huybreschts-K3}) that on a $K3$-surface these conditions
imply $\phi_1 \circ \phi_2^{-1}={\rm id}$ but we are not  aware of a
general argument for an arbitrary holomorphic-symplectic K\"ahler manifold.
\end{rmk}

We can now prove Theorem \ref{thm:uniquerigidity}, which we restate for
convenience.

\begin{thm} \label{thm:uniquerigidity2}  (cf. Theorem \ref{thm:uniquerigidity})
Let $(M^{4n}, g, I, J)$ be a nondegenerate generalized K\"ahler manifold.  Any
two solutions $(g_i,I,J_i)$, $i=1,2$ of the generalized K\"ahler Calabi-Yau
equation in the $\Omega$-Hamiltonian deformation class agree, and moreover
define a hyper-K\"ahler structure.
\begin{proof} This is restating Propositions \ref{p:CY} and \ref{p:uniqueness}.
\end{proof}
\end{thm}

\section{A Formal GIT picture} \label{s:GIT}

In this section we establish a formal moment map interpretation of the
generalized K\"ahler Calabi-Yau equation.  A closely related problem was
addressed in \cite{fine}, giving a formal moment map interpretation of the
problem of prescribing the volume form of symplectic structures within a given
de Rham class. A similar formal framework concerning the problem of finding
constant scalar curvature K\"ahler metrics within a given K\"ahler class is
given in \cite{fujiki, do},  and has been extended to the generalized K\"ahler
case in \cite{boulanger, Gotomoment} (inspired from the setting in
\cite{gauduchon-GIT}).

\subsection{Setup}

Let $(g, I, J)$ be a nondegenerate generalized K\"ahler metric, denote the
holomorphic-symplectic
structures by $\Omega_I$ and $\Omega_J$, and denote by
\begin{align} \label{Psidef}
\Psi_1 :=  I\Omega = {\rm Im}(\Omega_I), \qquad \Psi_2:= J\Omega = {\rm
Im}(\Omega_J)
\end{align}
the closed imaginary parts of $\Omega_I$ and $\Omega_J$.  Furthermore, let
$\alpha=[\Omega_I]$  and $\beta=[\Omega_J]$ be the corresponding deRham classes
in $H^2(M, \mathbb C)$.  Now set
\begin{align} \label{GKspacedef}
\mathcal{GK}_{\alpha,\beta} = \left\{ \mbox{GK triples } (g',I',J')\ |\
[\Omega_{I'}] \in \alpha,\ \Omega_{J'} \in [\beta] \right\}.
\end{align}
As the considerations in this section are purely formal, we endow ${\mathcal
GK}_{\alpha,\beta}$ with the $C^{\infty}$ topology.

By using Lemmas~\ref{l:symplectic-forms} and \ref{c}, we can interpret paths in
$\mathcal{GK}_{\ga,\gb}$ as smooth families of pairs of symplectic forms.  In
particular, consider a smooth family $\Omega_j^t=(\Omega^t +\i\Psi_j^t), j=1,2$
of closed complex-valued $2$ forms $\Omega^t_j$ on $M^{4n}$,  such that for all
$t$
\begin{enumerate}
\item[(a)] $\Omega^t$ and $\Psi^t_j$ are real symplectic $2$-forms and 
$[\Omega^t_1] = \alpha, \ \ [\Omega^t_2]=\beta$.
\item[(b)] the endomorhisms  $I_t:= -(\Psi^t_1)^{-1} \Omega^t$  and $J_t :=
-(\Psi^t_2)^{-1} \Omega^t$ define almost-complex structures;
\item[(c)] the $(1,1)$-part with respect to $I_t$ of  $-\Psi^t_2$ is positive
definite.
\end{enumerate}
It is observed in \cite{bande-kotschick} that  $I_t$ and $J_t$ are automatically
 integrable  (as an easy consequence of the closeness of $\Omega^t_j$), and
thus, by Lemma~\ref{l:symplectic-forms}, $(\Omega_j^t, j=1,2)$  give rise to a
smooth curve in $\mathcal{GK}_{\alpha, \beta}$. Conversely, by Lemma~\ref{c} 
any path in
$\mathcal{GK}(\alpha, \beta)$ has this form. 

By Moser's lemma with respect to the path of cohomologous symplectic
forms $\Omega^t$,  we can pull-back $(\Omega^t_j, j=1,2)$ by an  isotopy of
diffeomorphisms and assume that ${\rm Re}(\Omega^t_j)= \Omega$  is a fixed
symplectic form on $M$. We shall thus be interested, without loss, in a
restricted space, namely
\begin{align}
\mathcal{GK}_{\ga,\gb}(\Omega) = \left\{ (g',I',J') \in \mathcal{GK}_{\ga,\gb}\
|\ \Omega={\rm Re}(\Omega_I)={\rm Re}(\Omega_J) \right\}.
\end{align}
Note that points of $\mathcal{GK}_{\alpha, \beta}(\Omega)$ are equivalently
parametrized by pairs of complex structures $(I,J)$ (which in turn determine
$\Psi_1$ and $\Psi_2$ via (b)). Also, to simplify some of the discussion below,
we make
a further restriction, namely we set
\begin{align}
\mathcal{GK}^K_{\alpha, \beta}(\Omega) = \left\{ (g',I',J') \in \mathcal
{GK}_{\ga,\gb}(\Omega)\ | \ I', J' \mbox{ K\"ahler} \right\}.
\end{align}
The next result shows the openness of the orbits of the natural
$\Omega$-Hamiltonian action in this setting.

\begin{lemma}\label{l:orbit} Let $(M^{4n}, I)$ be a hyper-K\"ahler manifold and
suppose $(g, I, J)$ is a generalized K\"ahler metric in $\mathcal{M}$ an open
path connected
subset of
$\mathcal{GK}_{\ga,\gb}^K(\Omega)$.  Let $\mathcal{O}_I$ and $\mathcal{O}_J$
denote the orbits of the 2-forms ${\rm Im}(\Omega_I)$ and ${\rm Im}(\Omega_J)$
under the (right) action of ${\rm Ham}(M,\Omega)$. Then,  in the $C^{\infty}$
topology, $\mathcal{M}$ is an open subset of $\mathcal{O}_I\times
\mathcal{O}_J$.  If furthermore $b_1(M) = 0$, then $\mathcal{M}$ is finitely
covered by an open subset in ${\rm Ham}(M,\Omega) \times {\rm Ham}(M,\Omega)$.
\end{lemma}
\begin{proof} Any path in $\mathcal{M}$ starting at $(g,I,J)$ is determined by a
smooth family $(\Psi^t_j, j=1,2)$ of symplectic forms satisfying  the conditions
(a), (b) and (c) above with $\Omega^t=\Omega$. Thus, the $2$-forms $\gamma_j^t
:= \tfrac{\partial}{\partial t} \Psi^t_j$ are exact  (by (a)).  Writing
$\gamma_j^t = da_j^t$ for some $1$-forms $a_j^t$,  the fact that
$(\Omega +\i\Psi_j^t)^{2n}$ is a non-degenerate $(2n,0)$ form  and $(\Omega + i
\Psi_j^t)^{2n+1}=0$ imply that $da_1^t$ (resp. $da_2^t$) is of type
$(1,1)$ with respect to $I_t$ (resp. $J_t$). By the $\partial \bar
\partial$-Lemma for $(1,1)$ forms (which holds for each of the complex
structures $I_t$ and $J_t$ for any element in $\mathcal{M}$) we conclude that
there are unique smooth functions $f_{j}^t$, normalized by $\int_M f_{j}^t
\Omega^{2n}=0$, such that $\gamma_1^t = d I_t d f_{1}^t$ and $\gamma_2^t= d J_t
d f_{2}^t$.  Observe that Hodge theory with respect to some K\"ahler metric
implies that
$f_{j}^t$ vary smoothly in $t$.  We now apply Moser's lemma to each of the
families $\Psi_1^t$ and $\Psi_2^t$.
It shows that 
$$\Psi_1^t = (\phi_1^t)^*({\rm Im}(\Omega_I)); \ \  \Psi_2^t =
(\phi_2^t)^*({\rm Im}(\Omega_J)),$$
where $\phi_j^t$ are the flows of the time-dependent  vector fields $X_1^t
=-(\Psi_1^t)^{-1}(I_t d f_1^t)$ and $X_2^t = -(\Psi_2^t)^{-1}(J_t d f_2^t)$. We
claim that $X_j^t$ are Hamiltonian with respect to $\Omega$. Indeed, using (b),
we have
$$\imath_{X_1^t}(\Omega)= -\Psi_1^t(I_t X_1^t, \cdot) = -\Psi_1^t(X_1^t, I_t
\cdot) = I_t \Psi_1^t(X_t) = df_1^t,$$
and similarly for $X_2^t$, finishing the first claim.

The second claim follows from Remark~\ref{r:uniqueness} above, using that
$b_1(M)=0$.
\end{proof}

\begin{rmk}\label{M}
The setting above extends naturally to the general (not
necessarily K\"ahler) case. On any holomorphic-symplectic manifold $(M^{4n}, I,
\Omega_I=\Omega + iI\Omega)$,  the subgroup of ${\rm Ham}(M, \Omega)$ leaving
${\rm Im}(\Omega_I)=I\Omega$ invariant  is a closed subgroup of ${\rm Aut}(M,
I)$ with Lie algebra  identified with the  $\Omega$-Hamiltonian vector fields
$X= -\Omega^{-1}(df)$ with ${\mathcal L}_X (I\Omega) =0$, i.e. satisfying
$dIdf=0$. It follows that the latter is trivial when $M$ is compact, i.e. the
stabilizer of a generalized K\"ahler structure $(g, I, J)$ under the local
action of ${\rm Ham}(M,\Omega)\times {\rm Ham}(M,\Omega)$ is a discrete
subgroup.  Thus, generalizing the setting of Lemma~\ref{l:orbit} above, we let 
${\mathcal M} \subset \mathcal {GK}_{\alpha,\beta}(\Omega)$ be a
path-connected component of an orbit for the local action of ${\rm
Ham}(M,\Omega)\times {\rm Ham}(M,\Omega)$ on ${\mathcal
{GK}}_{\alpha,\beta}(\Omega)$. The proof of Lemma \ref{l:orbit} shows that any
tangent vector of $\mathcal{M}$
at a point $(I, J)$ is identified with a pair of exact forms $(-dIdf, dJdg)$
for uniquely determined $\Omega$-normalized smooth functions $f,g$. Thus, we
have an identification
\begin{equation}\label{identification}
T_{I, J} {\mathcal M} \cong \mathcal{C}^{\infty}_0(M) \oplus
\mathcal{C}^{\infty}_0(M),
\end{equation} which we use throughout this section. 
\end{rmk}

\subsection{The Aubin-Yau functional}

In this subsection we establish a variational characterization of the
generalized K\"ahler Calabi-Yau equation, in analogy with the Aubin-Yau
functional for the classical Calabi-Yau equation.  To begin we recall some
fundamental aspects of Hamiltonian actions on symplectic manifolds.  The group
$\mathcal{H}:={\rm Ham}(M, \Omega)$ can be thought as is an infinite dimensional
analog of a compact Lie group. Indeed, it is simple by a result of
Banyaga~\cite{banyaga}. Furthermore,  its Lie algebra is  identified with the
space $C^{\infty}_0(M)$ of smooth functions on $M$ with zero mean with respect
to $\Omega^{2n}$,  endowed with the Poisson bracket $\{f, g\} = \langle
\Omega^{-1}, df\wedge dg \rangle:= \tfrac{1}{2}{\rm tr}\Big(\Omega^{-1} \circ
(df\wedge dg)\Big)$. Then, the $L^2$-product  
\begin{equation}\label{g}
{\bf g}(f, g) = \tfrac{1}{(2n)!}\int_M f g \ \Omega^{2n}
\end{equation}
defined for any $f, g \in C^{\infty}_0(M)$ gives rise to an ${\rm ad}$-invariant
inner-product on ${\rm Lie}(\mathcal{H})$, and thus to an ${\rm Ad}$-invariant
Riemannian metric on $\mathcal{H}$, a property characterizing the finite
dimensional compact simple Lie groups. It is known that (see
e.g.~\cite{khesin-et-al}) the geodesics with respect to ${\bf g}$ are the flows
of time-independent Hamiltonian vector fields on $(M,\Omega)$.

As above (see Remark~\ref{M}), $\mathcal{M} \subset
\mathcal{GK}_{\ga,\gb}(\Omega)$  will denote a
path connected component of an orbit  for the local action of
$\mathcal{G}=\mathcal{H} \times \mathcal{H}$ on $\mathcal{GK}_{\alpha,
\beta}(\Omega)$ (the  whole group $\mathcal{G}$ acts only locally on
$\mathcal{GK}_{\alpha, \beta}(\Omega)$ because of the open condition (c)) and
notice that  the diagonal action of $\mathcal{H}$ on ${\mathcal G} = \mathcal{H}
\times \mathcal{H}$ descends to a well-defined global action of $\mathcal{H}$ on
$\mathcal{M}$, by pulling
back each generalized K\"ahler structure $(g, I, J)$ via the natural right
action of the
diffeomorphisms on $M$.  We also want to emphasize that
our formal manifold ${\mathcal M}$ is locally a subset of the 
infinite dimensional Lie group ${\mathcal G}$, thus we have a non-abelian
version of the familiar Calabi-Yau setting. 

Using the identification \eqref{identification}, for any pair of normalized
functions $(f, g) \in  \mathcal{C}^{\infty}_0(M)
\oplus \mathcal{C}^{\infty}_0(M) = {\rm Lie}(\mathcal{G})$, we have a
canonically associated vector field on $\mathcal{M}$, defined by  
\begin{equation}\label{fundamental-fields}
(f, g)_{I,J} := (-dIdf, -dJdg).
\end{equation}
Notice that $(f, g)$ are the  vector fields  induced by the local (right) action
of $\mathcal{G}$ on $\mathcal{M}$ and will play a key role in the computations
below. We shall use that such vector fields span each tangent space of
$\mathcal{M}$, and 
\begin{equation}\label{lie-bracket}
[(f_1, g_1), (f_2,g_2)]_{\mathcal{M}}= (\{f_1,f_2\}, \{g_1, g_2\}), 
\end{equation}
where $[\cdot, \cdot]_{\mathcal M}$ stands for the Lie bracket of vector fields
on $\mathcal{M}$,  and $\{\cdot, \cdot\}$ stands for the Poisson bracket with
respect to $\Omega$.

\begin{defn} Given the setup above, we define the $1$-form $\boldsymbol{\sigma}$
on ${\mathcal M}$, defined at a point $(I, J)$ and a tangent vector 
$(f, g) \in C^{\infty}_0(M)\oplus C^{\infty}_0(M) \cong T_{I,J}({\mathcal M})$
by
\begin{equation}\label{sigma}
\boldsymbol{\sigma}_{I,J} (f, g):= \tfrac{1}{(2n)!} \int_M (f-g)\Big((F_+)^{2n}
-
e^{\lambda} (F_{-})^{2n})\Big),
\end{equation}
where, we recall,  $F_{\pm}= 2(\pm \Psi_1 - \Psi_2)=-2g (I \pm J)^{-1}$ are the
real symplectic forms taming both $I$ and $J$, and $\gl$ is the cohomological
constant associated to $\mathcal{M}$ (cf. Lemma \ref{lambdalemma}).
\end{defn}

\begin{lemma}\label{l:closed} Given the setup above,
$\boldsymbol{\sigma}$ is invariant under the diagonal action of ${\rm Ham}(M,
\Omega)$ on $\mathcal{M}$ and is closed.
\end{lemma}
\begin{proof} The claim of invariance is obvious.  To show $\boldsymbol{\sigma}$
is closed, by using
\eqref{lie-bracket} it is enough to show that for any fundamental vector fields
$(f_1,g_1)$ and $(f_2,g_2)$ on $\mathcal{M}$ defied via
\eqref{fundamental-fields}, we have
\begin{equation}\label{closed}
\begin{split}
0 =&\ d\boldsymbol{\sigma}\big((f_1,g_1), (f_2, g_2)\big)\\
=&\ (f_1, g_1)\cdot
\boldsymbol{\sigma} (f_2, g_2) - (f_2, g_2)\cdot \boldsymbol\sigma (f_1, g_1) -
\boldsymbol\sigma\big(\{f_1,f_2\}, \{g_1, g_2\}\big).
\end{split}
\end{equation}
The induced flow $\phi^t$  by $(f, g)$ on $\mathcal{M}$ is defined at  a
point $(\Psi_1, \Psi_2)$ by 
$$\phi^t \cdot (\Psi_1, \Psi_2)= ((\phi^t_1)^*(\Psi_1),
(\phi^t_2)^*(\Psi_2)),$$ where $\phi^t_1$ is the flow of
$-\Omega^{-1}(df)$ and $\phi^t_2$ is the flow of $-\Omega^{-1}(dg)$. It
follows that the induced symplectic forms $(F_{\pm})_t$ satisfy along the flow
$$\tfrac{\partial}{\partial t} (F_{\pm})_t = \mp 2dI d f + 2dJ dg.$$
From the above relation, using integration by parts and \eqref{Fdefs}, we 
calculate
\begin{equation*}
\begin{split}
(f_1 , g_1)\cdot \boldsymbol\sigma (f_2, g_2) & = \int_M (f_2 -
g_2)\tfrac{\partial}{\partial t}\Big|_{t=0}\Big(\tfrac{(F_+)^{2n} - e^{\lambda}
(F_-)^{2n})}{(2n)!}\Big) \\
=&\ \tfrac{1}{2}\int_M
\Big\langle (I+J)(df_2 - dg_2), (I+J)(df_1-dg_1)\Big\rangle_g
\tfrac{(F_+)^{2n}}{(2n)!}\\
&\ +
\tfrac{e^{\lambda}}{2} \int_M \Big\langle (I-J)(df_2 - dg_2),
(I-J)(df_1-dg_1)\Big\rangle_g \tfrac{(F_{-})^{2n}}{(2n)!}\\
&\ + \tfrac{1}{2}\int_M
\Big\{ f_1+g_1, f_2 -g_2\Big\} \Big(\tfrac{(F_+)^{2n}  -
e^{\lambda}(F_{-})^{2n}}{(2n)!}\Big).
\end{split}
\end{equation*}
The identity \eqref{closed} then follows easily. \end{proof}
 
 \begin{rmk}\label{r:exactness} Notice that the $1$-form $\boldsymbol\sigma$ restricts to zero on each orbit of the
diagonal action of  $\mathcal{H}={\rm Ham}(M,\Omega)$ on $\mathcal{M}$, since
the tangent vectors to this orbit are generated by the fundamental vector fields
$(f, f)$. Moreover, $\sigma$ vanishes at $(I,J) \in \mathcal{M}$ if and only if
$(F_+)^{2n} = e^{\lambda} (F_{-})^{2n}$, i.e. iff $(I, J)$ is
hyper-K\"ahler GK structure in $\mathcal{M}$, (cf. Theorem
\ref{thm:uniquerigidity}).    If we can find a primitive  ${\bf F}$  of $\boldsymbol\sigma$ on
the space $\mathcal{M}$, then ${\bf F}$ will define a functional which is invariant under the diagonal action of ${\rm Ham}(M, \Omega)$,  and its critical points will parametrize the hyper-K\"ahler GK structures in
$\mathcal{M}$,  modulo the diagonal (isometric) action of ${\rm Ham}(M,\Omega)$. 
However, as $\mathcal{H}={\rm Ham}(M,\Omega)$ may in principle have a
complicated topology (in particular,  $\pi_1(\mathcal{H})\neq \{1\}$ in general)
one needs to define ${\bf F}$ on  the universal cover ${\widetilde {\mathcal M}}$ of
$\mathcal{M}$ (and of $\mathcal{M}/ \mathcal{H}$).  
\end{rmk}

\begin{prop} \label{AubinYauprop}  Given the setup above, there exists a
functional
\begin{align*}
{\bf F} : \til{\mathcal M} \to \mathbb R
\end{align*}
such that $\gd {\bf F} = \pi^* {\boldsymbol \sigma}$, where $\pi : \til{\mathcal
M} \to \mathcal M$ is the canonical projection.  In particular, the critical
points of $\bf F$ correspond to hyper-K\"ahler metrics.
\begin{proof} This follows directly from the fact that ${\boldsymbol \sigma}$ is
closed via formal path integration and Remark~\ref{r:exactness}.
\end{proof}
\end{prop}

This functional ${\bf F}$ is a natural analogue of the Aubin-Yau
functional in the classical (abelian) setting.  However, to avoid the use of
the universal cover, we can also directly define path integrals of $\boldsymbol
\sigma$.  To define the relevant paths we turn to the more geometrically natural
space $\mathcal{M}/\mathcal{H}$, which can be thought of as an infinite
dimensional ``orbifold'' because of Lemma \ref{l:orbit}.  

\subsection{The geodesic space ${\mathcal M}/{\mathcal H}$}

We shall use the smooth  identification 
$${\mathcal G}/{\mathcal H} \cong {\mathcal H},$$
 given by the map $[\phi_1, \phi_2] \to \phi_1^{-1} \phi_2$ with
inverse
 \begin{equation}\label{isom}
 \begin{array}{ccc}
 {{\mathcal H}} & \cong & {{\mathcal G}}/{{\mathcal H}} \\
  \phi &\mapsto & [{\rm id}, \phi].
  \end{array}
  \end{equation}
In view of Lemma~\ref{l:orbit},   we can identify $\mathcal{M}/\mathcal{H}$ with
the path-connected component in ${\mathcal O}_J$ of $\tilde J$'s  such that $(I,
\tilde J)$ satisfy the condition (c). This is precisely the setting of
Section~\ref{s:GCY} above. 

Recall from Lie theory that for any compact simple
Lie group $H$, the pair $(H\times H, H)$ with $H$ acting diagonally defines a
symmetric pair, i.e. $H$ can be viewed as a Riemannian-symmetric space with
respect to any ${\rm Ad}$-invariant Riemannian metric (i.e. defined by negative
multiple of the Killing form on ${\rm Lie}(H)$). Furthermore, in terms of the
isomorphism \eqref{isom}, the flows of the left-invariant vector fields of $H$
are the geodesics of the Riemannian-symmetric space $H \cong (H\times H)/H$.
Similarly, in our infinite dimensional setting,  the flows of the fundamental
vector fields $(0,g)$ are the geodesics with respect to the $L^2$ Riemannian
metric ${\bf g}$ on $\mathcal{H}$ (see \cite{khesin-et-al}). Thus, the flows of
the fundamental vector fields $(0,g)$ acting on $\mathcal{O}_J$ are the geodesics with respect to the $L^2$
Riemannian metric ${\bf g}$ defined on $\mathcal{O}_J$  (and also $\mathcal M/\mathcal H \subset \mathcal{O}_J$) by \eqref{g}.  This gives rise to a formal notion of ${\bf g}$-geodesic on
$\mathcal{M}/\mathcal{
H}$. With this background we define a corresponding functional.

\begin{defn} Given the setup above, fix an $\Omega$-normalized time independent
function $g \in
C^{\infty}(M)$, and define a
path $(\Psi_1^t=\Psi_1, \Psi_2^t= \phi_t^* \Psi_2)$ in $\mathcal M$,  where
$\phi_t$ is the flow of $-\Omega^{-1}(dg)$.  We thus  have
\begin{align*}
\dot{\Psi}_1^t = 0, \qquad \dot{\Psi}_2^t = -dJ_t dg.
\end{align*}
Along this path we define
 \begin{equation}\label{deltaF}
 \delta{\bf F}_g(t) := \boldsymbol{\sigma}_{I_t, J_t}(0, g) = -
\tfrac{1}{2n!}\int_M g \Big((F_+)_t^{2n} - e^{\lambda} (F_{-})_t^{2n})\Big).
 \end{equation} 
\end{defn}

As the next proposition shows, this functional is monotone, corresponding
formally to geodesic convexity of $\bf F$.

\begin{prop}\label{l:strict-convexity} The function $\delta{\bf F}_g(t)$ is
monotone nondecreasing along ${\bf g}$-geodesics of $\mathcal{M}/{\mathcal H}$.
\begin{proof} Using the calculation of Lemma \ref{l:closed}, noting that
$f_t \equiv 0$ and $g_t \equiv g$ we obtain
 \begin{equation*}
 \begin{split}
 \tfrac{\partial }{\partial t} \delta{\bf F}(t) =&\ \tfrac{1}{2}\int_M
\Big\langle
(I_t+J_t)dg, (I_t+J_t)dg \Big\rangle_{g_t} \tfrac{(F_+)_t^{2n}}{(2n)!}\\
 &\ + \tfrac{e^{\lambda}}{2} \int_M \Big\langle (I_t-J_t)dg,
(I_t-J_t)dg\Big\rangle_{g_t} \tfrac{(F_-)_t^{2n}}{(2n)!}\\
 =&\ \tfrac{1}{2}\int_M \brs{ (I_t+J_t)dg}^2 \tfrac{(F_+)_t^{2n}}{(2n)!} +
\tfrac{e^{\lambda}}{2} \int_M \brs{(I_t-J_t)dg}^2 \tfrac{(F_-)_t^{2n}}{(2n)!}\\
 \geq&\ 0,
\end{split}
\end{equation*}
as required.
\end{proof}
\end{prop}
We close by noting that Proposition \ref{l:strict-convexity} points towards the
uniqueness of a hyper-K\"ahler metric in $\mathcal{M}/\mathcal{H}$, shown in
Proposition~\ref{p:uniqueness}, and provides an alternative proof assuming
geodesic convexity of $\mathcal{M}/\mathcal{H}\subset \mathcal{O}_J$.

\subsection{Symplectic form and moment map}

In this subsection we define a symplectic structure on $\mathcal{M}$, and then
compute the moment map of our action.  As in the finite dimensional case, the
product  $\mathcal{G}=\mathcal{H} \times \mathcal{H}$ admits a natural
left-invariant almost-complex structure ${\bf I}$, defined on its Lie algebra
${\rm Lie}({\mathcal G})=C^{\infty}_0(M)\oplus C^{\infty}_0(M)$ by 
\begin{equation*}
{\bf I}(f, g) =(-g, f),
\end{equation*}
for all $f, g \in C^{\infty}_0(M) $. One can easily check that the left and
right diagonal action of $\mathcal{H}$ on ${\mathcal G}$ preserves ${\bf I}$,
however ${\bf I}$ need not be integrable.  The almost-complex structure ${\bf
I}$ on $\mathcal{G}$  induces  an almost-complex structure on $\mathcal{M}$,
which we still denote by ${\bf I}$.

\begin{defn} Given the setup above, let
\begin{align} \label{Idef}
{\bf I}_{I,J}(-dIdf, -dJdg) := (dIdg, -dJdf).
\end{align}
\end{defn}

With this in place, we can follow finite dimensional constructions to define a
symplectic form on $\mathcal M$ as well.  We give the definition, then
show that it is indeed symplectic in Lemma \ref{bigOsymp} below.
\begin{defn} Given the setup above, let
\begin{equation}\label{Omega-0}
{\bf \Omega} := d {\bf I} {\boldsymbol \sigma}.
\end{equation}
\end{defn}

\begin{lemma} \label{bigOsymp} Given the setup above, $\bf \Omega$ is closed and
tames $\bf I$, so is nondegenerate.

\begin{proof} By definition $\bf \Omega$ is exact, so is closed.  To show that
$\bf \Omega$ tames $\bf I$, we first note that by definition, on fundamental
vector fields $(f_1, g_1)$ and $(f_2,g_2)$ we have (see
\eqref{fundamental-fields} and \eqref{lie-bracket})
\begin{equation}\label{Omega-1}
\begin{split}
{\bf \Omega}_{I, J}\Big((f_1,g_1), (f_2, g_2)\Big) = &  d ({\bf I}{\boldsymbol
\sigma})_{I,J} \Big((f_1, g_1), (f_2, g_2)\Big) \\
                                                                              =
& (f_1, g_1)\cdot \boldsymbol{\sigma} (g_2, -f_2) - (f_2, g_2)\cdot
\boldsymbol\sigma (g_1, -f_1)\\
                                                                      &  -
\boldsymbol\sigma\big(\{g_1,g_2\}, -\{f_1, f_2\}\big).
\end{split}
\end{equation}
Using this and calculations as in the proof of Lemma \ref{l:closed} we obtain
\begin{equation}\label{Omega-2}
\begin{split}
{\bf \Omega}_{I, J}& \Big((f_1,g_1), (f_2, g_2)\Big)\\
=&\ \int_M \big\langle (I+J)df_1, (I+J)dg_2\big\rangle_g 
\tfrac{(F_+)^{2n}}{(2n)!} -\int_M \big\langle (I+J)df_2, (I+J)dg_1\big\rangle_g
\tfrac{(F_+)^{2n}}{(2n)!} \\
                                                                           & +
e^{\lambda} \int_M \big\langle(I-J)df_1, (I-J)dg_2\big\rangle_g
\tfrac{(F_{-})^{2n}}{(2n)!} - e^{\lambda}\int_M \big\langle (I-J)df_2,
(I-J)dg_1\big\rangle_g \tfrac{(F_{-})^{2n}}{(2n)!}\\
                                                                           &+
\int_M \Big(\{ f_1, g_2\}  - \{f_2, g_1\}\Big)\Big(\tfrac{(F_+)^{2n}  -
e^{\lambda} (F_{-})^{2n}}{(2n)!}\Big).
\end{split}
\end{equation}

This implies that
\begin{align*}
 {\bf \Omega}_{I, J} \Big((f,g), {\bf I} (f, g)\Big) =&\ {\bf \Omega}_{I, J}
\Big((f,g), (-g, f)\Big)\\
=&\ \int_M \brs{(I+J)df}^2_g \tfrac{(F_+)^{2n}}{(2n)!} + \int_M
\brs{(I+J)dg}^2_g
\tfrac{(F_+)^{2n}}{(2n)!} \\
& + e^{\lambda} \int_M \brs{(I-J)df}^2_g
\tfrac{(F_{-})^{2n}}{(2n)!} + e^{\lambda} \int_M \brs{(I-J)dg}^2_g
\tfrac{(F_{-})^{2n}}{(2n)!}\\
\geq&\ 0.
\end{align*}
This shows that ${\bf \Omega}$ tames ${\bf I}$, so that it is non-degenerate, as
required.
\end{proof}
\end{lemma}

We close by observing that the diagonal action of ${\rm \Ham}(M, \Omega)$ is
Hamiltonian with respect to ${\bf \Omega}$, and then compute the moment map.

\begin{prop} \label{momentmap} The diagonal action of ${\rm Ham}(M,\Omega)$  on
$\mathcal{M}$ is Hamiltonian with respect to the symplectic form  ${\bf
\Omega}$, with momentum map
$$\boldsymbol\mu(I, J) = 2\Big((F_+)^{2n} - e^{\lambda}
(F_{-})^{2n})\Big), $$
seen as an element of dual vector space of $C^{\infty}_0(M) \cong  {\rm
Lie}({\rm Ham}(M,\Omega))$ via integration over $M$.
\end{prop}

\begin{proof} 
Given the setup above, let
\begin{equation}\label{mu}
\boldsymbol\mu_f(I,J):= ({\bf I}\boldsymbol\sigma)_{I,J}(f, f) =\boldsymbol
\sigma (f, -f)=2\int_M f \Big((F_+)^{2n} - e^{\lambda}
(F_{-})^{2n})\Big)/(2n)!.
\end{equation}
Next, we note that \eqref{Omega-2} implies that ${\bf \Omega}$ is invariant
under the diagonal action of ${\mathcal H}$ on $\mathcal{M}$.  This also shows
that for each  fundamental vector field $(f, f)$ on $\mathcal{M}$ (manifestly
induced by the diagonal  action of $\mathcal{H}$), one has, using
\eqref{Omega-2} and the main calculation of 
Lemma \ref{l:closed},
\begin{align*}
{\bf \Omega}_{I, J}\Big((f,f), (f_2, g_2)\Big) = -d \Big(({\bf
I}\boldsymbol\sigma)_{I,J}(f, f)\Big)(f_2, g_2)= (f_2, g_2) \cdot \boldsymbol
\sigma(f,-f). 
\end{align*}
Lastly, we observe using Lemma \ref{l:closed} and
\eqref{mu}, that
\begin{align*}
(d{\boldsymbol \mu}_f)(g, g) = (g, g) \cdot {\boldsymbol \mu}_f = (g, g) \cdot
{\boldsymbol \sigma}(f, -f) = {\boldsymbol \mu}_{\{f,g\}}, 
\end{align*}
showing that $\boldsymbol \mu$ is $\mathcal H$-equivariant.  The proposition
follows. \end{proof}

\begin{rmk}\label{r:Kempf-Ness} We can  think of the
functional ${\bf F}$  on $\mathcal{M}$ (whenever it is defined) as a
``Kempf-Ness functional'', in the sense that it satisfies $({\bf I} d {\bf F})(f, f) =  {\boldsymbol \sigma}(f,
-f)={\boldsymbol\mu}_f$ (see \eqref{deltaF} and \eqref{mu}), and, considered as a functional on $\mathcal M /\mathcal H$, it is
strictly-convex along  the ${\bf g}$-geodesics of  $(\mathcal{M}/\mathcal{H}, {\bf g})$ by
Lemma \ref{l:strict-convexity}. Thus,  its only critical point on
$(\mathcal{M}/\mathcal{H}, {\bf g})$ corresponds to the $\mathcal{H}$-orbit of
the zero of $\boldsymbol\mu$. \end{rmk}

\section{Nondegenerate Generalized K\"ahler-Ricci flow} \label{flowsec}

In this section we show that in the setting of generalized K\"ahler structures
with pure spinors,
the generalized K\"ahler Ricci flow preserves the natural variation classes of
objects defined above.  In particular, we show
that the flow reduces to a family of $\Omega$-Hamiltonian diffeomorphisms, which
is generated by the Ricci potential.  We then show that the Ricci potential
itself evolves by the pure time-dependent heat equation along the flow, which
leads to a number of delicate a priori estimates along the flow.  Then we relate
the flow to the GIT picture, in particular showing convexity of $\bf F$,
finishing the proof of Theorem \ref{flowandGIT}.  We end by showing a general
result
showing that uniform equivalence of the time dependent metrics suffices to show
long time existence and convergence to hyper-K\"ahler for the flow.

\subsection{Background on GKRF}

In this subsection we review the construction of generalized K\"ahler Ricci flow
(GKRF) from \cite{STGK}.  To begin we review the pluriclosed flow \cite{ST2},
defined by
\begin{align} \label{PCF}
 \dt \gw =&\ - 2 (\rho_B)^{1,1}.
\end{align}
This equation can also be expressed using the curvature of the Chern connection.
 In \cite{ST2} we showed that this flow preserves the pluriclosed condition and
agrees with
K\"ahler-Ricci flow when the initial data is K\"ahler.
Moreover, the induced pairs of metrics and
Bismut torsions $(g_t,H_t)$ satisfy (\cite{STGK} Proposition 6.3), 
\begin{gather} \label{PCFmetricevs}
 \begin{split}
 \dt g =&\ - 2 \Rc^g + \tfrac{1}{2} \HH - \LL_{\theta^{\sharp}} g,\\
 \dt H =&\ \gD_g H - \LL_{\theta^{\sharp}} H,
\end{split}
\end{gather}
where $\HH_{X,Y} := \sum_{i,j=1}^{2m} H(X,e_i,e_j) H(Y,e_i,e_j)$, $\Rc^g$ is the
Ricci tensor of $g$.  

As explained in \cite{STGK}, with a generalized K\"ahler initial condition one
can unify the two pluriclosed flow lines given by the distinct pluriclosed
structures by removing the gauge terms to arrive at the
generalized
K\"ahler-Ricci flow system
\begin{gather} \label{RGflow}
 \begin{split}
\dt g = - 2 \Rc^g + \tfrac{1}{2} \HH, \qquad \dt H = \gD_g H,\\
\dt I = \LL_{\theta_I^{\sharp}} I, \qquad \dt J = \LL_{\theta_J^{\sharp}} J.
\end{split}
\end{gather}
In obtaining estimates for the flow, we need to use two different points of
view, each of which makes certain estimates possible.  Some estimates will use
the system (\ref{RGflow})
directly, which we will call a solution ``in the RG flow gauge.''  Other times
it is easier to work with pluriclosed flow directly, so we pull back the flow to
the fixed complex manifold $(M^{2n}, I)$. 
In other words by pulling back the entire system by the family of
diffeomorphisms $(\phi_t^I)^{-1}$ we return to pluriclosed flow on $(M^{2n},
I)$, which
encodes everything about the GKRF except the other complex
structure, which is given by a certain diffeomorphism pullback.  We will refer
to this point of view on GKRF as occurring ``in the $I$-fixed
gauge."  For concreteness, we record the evolution equations for the GKRF in the
$I$-fixed gauge,
\begin{gather} \label{Ifixed}
 \begin{split}
  \dt \gw_I =&\ - 2 (\rho_B)^{1,1}_I, \qquad \dt J = \LL_{\theta_J^{\sharp} -
\theta_I^{\sharp}}.
 \end{split}
\end{gather}

\begin{rmk} Two distinct Laplacians are relevant to the analysis to follow. 
First, we have the Riemannian Laplacian acting on functions
\begin{align*}
\gD_g f := \IP{\N^2 f, g}.
\end{align*}
Also, we will use the Chern Laplacian associated to a Hermitian structure
$(g,I)$, which takes the form
\begin{align*}
\gD^C_g f :=\langle dd^c_I f, \omega_I\rangle_g=\gD_g f -
\langle d f, \theta_I\rangle_g.
\end{align*}
We will also use the notation $\gD_g$ and $\gD_g^C$ respectively for the Levi-Civita and Chern Laplacians acting on sections of other vector bundles.  This formulas clarify an important point we will exploit in several places.  In
particular, the two Laplacians differ by the natural action of the Lee vector
field on the function.  This is also the relevant vector field relating the
solution (\ref{RGflow}) of the generalized K\"ahler-Ricci flow in the RG flow
gauge with the solution (\ref{Ifixed}) in the $I$-fixed gauge.  In particular,
evolution equations for scalar quantities expressed naturally using the
Riemannian Laplacian in the RG flow gauge will take an identical form except for
the use of the Chern Laplacian when converting to the $I$-fixed gauge.
\end{rmk}

\subsection{Nondegenerate case}

\begin{lemma}\label{l:flow} Suppose $(g_t, I, J_t)$ is a smooth solution of the
$I$-fixed gauge GKRF with nondegenerate initial data. Then, one
has the evolution equations
\begin{enumerate}
\item $\tfrac{\part}{\part t} \omega_I =  - 2 (\rho_B)_I^{1,1} = \left(dJd \Phi
\right)_I^{1,1},$
\item $\tfrac{\part}{\part t} \Omega_{I} = 0,$
\item $\tfrac{\part }{\part
t} \Omega_{J} = \i (\rho_B)_I$.
\end{enumerate}
\end{lemma}

\begin{proof} The first equation is an immediate consequence of (\ref{PCF}) and
Proposition \ref{ricci-potential}.
For the second equation it is enough to  show that  real part $\Omega_t =
g_t[I, J_t]^{-1}$ of $\Omega_I$, or, equivalently the real Poisson structure
$\gs_t = g_t^{-1}[I, J_t]$, is constant along the flow.  Using
\eqref{Ifixed} we have
\begin{equation} \label{eveqns10}
\begin{split} 
 \tfrac{\part}{\part t} \gs_t &=   g_t^{-1} [I, {\mathcal L}_{\theta_J^{\sharp}
-
\theta_I^{\sharp}} J] - 2 g_t^{-1} (\rho_B)_I^{1,1} I g_t^{-1} [I, J_t]\\
&= g_t^{-1} [I, {\mathcal L}_{\theta_J^{\sharp} - \theta_I^{\sharp}} J]  - 2
g_t^{-1} (\rho_B)_I^{1,1} I \gs_t \\
&= g_t^{-1} [I, {\mathcal L}_{\theta_J^{\sharp} - \theta_I^{\sharp}} J]  -
g_t^{-1} \big((\rho_B)_I I + I (\rho_B)_I)\gs_t\\
&= g_t^{-1} [I, {\mathcal L}_{\theta_J^{\sharp} - \theta_I^{\sharp}} J -
(\rho_B)_I \gs_t],
\end{split}
\end{equation}
where for the last equality we have also used that $\gs_t$ anti-commutes with
$I$. To show that this in fact vanishes, we compute
\begin{equation}\label{key2}
\begin{split}
({\mathcal L}_{\theta_J^{\sharp}- \theta_I^{\sharp}} J) \Omega &= \mathcal
L_{\theta_J^{\sharp}- \theta_I^{\sharp}} (J \Omega) -  J {\mathcal
L}_{\theta_J^{\sharp}- \theta_I^{\sharp}}\Omega\\
&= d \Big(J \Omega(\theta_J^{\sharp} - \theta_I^{\sharp})\Big) - J d
\Big(\Omega(\theta_J^{\sharp} - \theta_I^{\sharp})\Big) \\
&= -\tfrac{1}{2} d J d \Phi  -  \tfrac{1}{2} J dd \Phi \\
&=  -\tfrac{1}{2} d J d\Phi\\
&= (\rho_B)_I,
\end{split}
\end{equation}
where for the second equality we used that both $\Omega$ and $J\Omega$ are
closed,
for the third equality we used Lemma~\ref{l:extension} whereas  the fifth
equality follows from Proposition~\ref{ricci-potential}.  Plugging (\ref{key2})
into (\ref{eveqns10}) yields (2).

Lastly, to establish (3), using that $\tfrac{\part}{\part t} \Omega_t=0$ as just
established, equations (\ref{Ifixed}) and \eqref{key2} then imply
\begin{gather}
 \begin{split}
\tfrac{\part}{\part t} \Omega_{J} = \i (\tfrac{\part}{\part t} J) \Omega =
\i \left({\mathcal L}_{\theta_J^{\sharp}-\theta_I^{\sharp}} J \right)\Omega = \i
(\rho_B)_I,
 \end{split}
\end{gather}
as required.
\end{proof}

With this lemma in place we can establish that the GKRF evolves by
$\Omega$-Hamiltonian diffeomorphisms.

\begin{prop}\label{p:Hamiltonian-flow} Suppose $(g_t, I, J_t)$ is a smooth
solution of the $I$-fixed gauge GKRF with nondegenerate initial data.
Let $\phi_t$ denote the flow of the time dependent, $\Omega$-Hamiltonian
vector field $X_{t}:=-\tfrac{1}{2} \gs d \Phi_t$. 
Then the induced family of generalized K\"ahler structures $(g_{\phi_t},
I, J_{\phi_t})$ obtained via Proposition \ref{nondegvariations} coincides
with
$(g_t, I, J_t)$.
\end{prop}
\begin{proof}  By Lemma~\ref{l:extension}, $X_t =
(\theta^t_J-\theta^t_I)^{\sharp}$,  where $\theta^t_I$ and
$\theta^t_J$ are the Lee forms along the GKRF and $\sharp$ denotes $g_t^{-1}$. 
It thus follows that
$$\tfrac{\part} {\part t} (J_{\phi^u_t}  - J_t)= 0,$$
showing that  $J_{\phi^u_t}  =J_t$ as they equal $J$ at $t=0$. According to
Lemma~\ref{l:flow}, and by a computation identical to \eqref{key2},  the
corresponding symplectic $2$-forms $F_t=2({\rm Im}(\Omega_I-\Omega_{J_t}))$
and $F_{\phi_t} = 2({\rm Im}(\Omega_I -(\phi_t)^*\Omega_{J}))$
satisfy
$$\tfrac{\part}{\part t} F_t =  \tfrac{\part}{\part t} F_{\phi_t
} =  d J_t d \Phi_t, $$
so that $F_t = F_{\phi_t}$.  Taking $(1,1)$-part with respect
to $I$ gives $g_t= g_{\phi_t}$, as required.
\end{proof}

\subsection{A priori estimates}

In this subsection we derive a priori estimates associated to the Ricci
potential along solutions
to GKRF with nondegenerate initial data.  First we show that in these
settings the associated Ricci potential satsifies the pure time dependent heat
equation.  This remarkably simple evolution equation can be exploited to obtain
decay of the gradient of the Ricci potential.  

\begin{prop}\label{nondegspinorev} Let $(M^{4n}, g_t, I, J_t)$ be a solution
to GKRF in the $I$-fixed gauge with nondegenerate initial data.  Let $\Phi_t =
\log \tfrac{F_+^{2n}}{F_-^{2n}}$ denote the
associated family of Ricci potentials.  Then
\begin{align*}
\left(\dt -\gD^C_{g_t} \right) \Phi =&\ 0.
\end{align*}
\begin{proof} Let $(F_{\pm})_t =  2(\pm {\rm Im}(\Omega_{I}) - {\rm
Im}(\Omega_{J_t}))=-2g_t(I\pm J_t)^{-1}$ denote the two symplectic forms given
by
the construction of Lemma~\ref{l:non-degenerate-symplectic}.   Lemma
\ref{l:flow} yields
\begin{equation}\label{omega-pm}
\tfrac{\part}{\part t} F_{\pm} =  dJd \Phi.
\end{equation}
On the other hand, it follows by the arguments in
Lemma~\ref{l:non-degenerate-symplectic} that
\begin{equation}\label{key-5}
(F_{\pm})^{2n} =2^{n}(\det(I\mp J))^{\tfrac{1}{2}}{\rm Re}(\Omega_I)^{2n},
\end{equation}
so that, using Lemma \ref{l:flow} part (1) again, we derive from
\eqref{omega-pm}
\begin{equation}\label{key-6}
\begin{split}
\tfrac{\part}{\part t} (F_{\pm})^{2n}/(2n!) &= (dJd \Phi)\wedge
(F_{\pm})^{2n-1}/(2n-1)!\\
&= \tfrac{1}{2} {\rm tr}\Big((d J d \Phi)\circ (F_{\pm})^{-1}
\Big)\Big((F_{\pm})^{2n}/(2n)!\Big)\\
&= \tfrac{1}{2} \big\langle dJd \Phi, \gw_I \pm \gw_J \big\rangle_g
\Big((F_{\pm})^{2n}/{(2n)!}\Big).
\end{split}
\end{equation}
Together with \eqref{key-5}, this implies
\begin{equation}\label{key-7}
\tfrac{\part}{\part t} \Phi = \langle d J d \Phi, \gw_J \rangle_g =
\gD_g \Phi - \langle d \Phi, \theta_J\rangle_g =
\gD_g \Phi - \langle d \Phi, \theta_I\rangle_g = \gD^C_g \Phi,
\end{equation}
where the second equality follows easily from \eqref{DF} and the third
equality follows from Lemma \ref{l:extension} and the fact that $[I,J]$ is skew.
\end{proof}
\end{prop}

\begin{cor}\label{c:potential-bound} Let $(M^{2n}, g_t, I_t, J_t)$ be a solution to GKRF in the
nondegenerate case.  Then
\begin{align*}
 \sup_{M \times [0,T)} \brs{\Phi} \leq \sup_{M \times \{0\}} \brs{\Phi}.
\end{align*}
\begin{proof} This follows immediately from Proposition \ref{nondegspinorev} and
the maximum principle.
\end{proof}
\end{cor}

Next we exploit the simple evolution equation of (\ref{nondegspinorev}) to
obtain a gradient estimate for the Ricci potential.  To begin we
recall a basic fact about solutions to the heat equation with
background flowing along the $(g,H)$ equations of (\ref{RGflow}).

\begin{lemma} \label{gradientev} (\cite{SNDG} Lemma 4.3) Let $(M^n, g_t, H_t)$
be a solution to the $(g,H)$ evolution equations of
(\ref{RGflow}), and let $\phi_t$ be a solution to
\begin{align*}
 \dt \phi =&\ \gD_{g_t} \phi.
\end{align*}
Then
\begin{align*}
 \tfrac{\del}{\del t} \brs{\nabla \phi}^2 =&\ \gD \brs{\N \phi}^2 - 2
\brs{\N^2 \phi}^2 - \tfrac{1}{2}
\IP{\HH, \N \phi \otimes \N \phi}.
\end{align*}
\end{lemma}

\begin{prop}  \label{gradspinorev} Let $(M^{2n},
g_t, I_t, J_t)$ be a solution to
generalized K\"ahler
Ricci flow with nondegenerate initial data.  Let $\Phi_t$ be the
associated family of Ricci potentials.  Then
 \begin{align}
  \left(\dt - \gD_{g_t} \right) \brs{\N \Phi}^2 =&\ - 2 \brs{\N^2 \Phi}^2 -
\tfrac{1}{2}
\IP{\HH, \N \Phi \otimes \N \Phi}.
 \end{align}
\begin{proof}  This follows directly from
Proposition
\ref{nondegspinorev} and Lemma
\ref{gradientev}.
\end{proof}
\end{prop}

\begin{prop}  \label{gradspinordecay} Let $(M^{2n},
g_t, I_t, J_t)$ be a
solution to generalized K\"ahler
Ricci flow with nondegenerate initial data.  Let $\Phi_t$ be the
associated family of Ricci potentials.  Then
\begin{align*}
 \sup_{M \times \{t\}} \brs{\N \Phi}^2 \leq t^{-1} \left( \sup_{M \times \{0\}}
\brs{\Phi}^2 \right).
\end{align*}
\begin{proof} Let
\begin{align*}
W = t \brs{\N \Phi}^2 + \Phi^2.
\end{align*}
Combining Proposition \ref{nondegspinorev} with Proposition \ref{gradspinorev}
we see
\begin{align*}
\left(\dt - \gD_{g_t} \right) W =&\ -2 t \brs{\N^2 \Phi}^2 - \tfrac{t}{2}
\IP{\HH, \N \Phi \otimes \N \Phi} - \brs{\N \Phi}^2.
\end{align*}
As the tensor $\HH$ is positive definite, we conclude that $W$ is a subsolution
to the heat
equation, and so we conclude from the maximum principle that
\begin{align*}
\sup_{M \times \{t\}} W \leq \sup_{M \times \{0\}} W = \sup_{M \times \{0\}}
\Phi^2.
\end{align*}
The result follows upon rearranging this.
\end{proof}
\end{prop}

\subsection{GIT framework and GKRF}

In this subsection we observe here monotonicity of the Aubin-Yau differential
$\boldsymbol{\sigma}$ under the
generalized
K\"ahler-Ricci flow, and record the proof of Theorem \ref{flowandGIT}.

\begin{prop}\label{p:energy} Let $(M^{2n},
g_t, I, J_t)$ be a
solution to generalized K\"ahler
Ricci flow in the $I$-fixed gauge with nondegenerate initial data.  Then
\begin{equation}\label{evolution-functional}
\tfrac{d^2}{dt^2} {\bf F}(I,J_t) = \tfrac{d}{dt} {\boldsymbol \gs}_{I, J_t}
(0,\Phi_t) = \int_M \brs{\N \Phi_t}^2_g
\Big((F_+)_t^{2n} + e^{\lambda} (F_-)_t^{2n} \Big) / (2n)!.
\end{equation}
\end{prop}
\begin{proof} By Proposition~\ref{p:Hamiltonian-flow}, the GKRF is the
Hamiltonian flow of $X_t= -\tfrac{1}{2}\gs^{-1}d\Phi_t$.  Using this and the
definition of $\bf F$, the first equation follows directly.  For the second,
following a
calculation similar to Proposition \ref{l:strict-convexity} yields
 \begin{equation}\label{2-1}
 \begin{split}
\tfrac{d}{dt} {\boldsymbol \gs}_{I,J_t}(0,\Phi_t) =&\ \tfrac{1}{4} \int_M
\brs{(I
+ J)d \Phi}^2 \tfrac{F_+^{2n}}{(2n)!} + \tfrac{e^{\lambda}}{4} \int_M \brs{(I -
J)
d \Phi}^2 \tfrac{F_-^{2n}}{(2n)!}\\
& - \int_M \tfrac{\partial}{\part t} \Phi_t \left(F_+^{2n} - e^{\gl} F_-^{2n}
\right) / (2n)!.
\end{split}
\end{equation}
We use Lemma~\ref{l:flow} to compute the last term.  To this end,
as $F_+= -2(I+J)^{-1}g$,  we have $(F_+)^{2n}/(2n)! =  2^{2n} \left(\det
(I+J)\right)^{-\tfrac{1}{2}} dV_g$, and we compute using \eqref{2} and
Lemma~\ref{l:extension}
\begin{equation}\label{2-2}
\begin{split}
\int_M \gD_g \Phi \tfrac{(F_{+})^{2n}}{(2n)!} =&\ 2^{2n} \int_M \gD_g \Phi
\left( \det (I+J)\right)^{-\tfrac{1}{2}} dV_g \\
=&\ \tfrac{1}{2} \int_M \big\langle d \Phi, d \log \det (I+J) \big\rangle_g
\tfrac{(F_+)^{2n}}{(2n)!}\\
=&\ \tfrac{1}{2}\int_M \big\langle d \Phi, \theta_I + \theta_J\big\rangle_g
\tfrac{(F_+)^{2n}}{(2n)!} + \tfrac{1}{4}\int_M \big\langle d \Phi, (I-J)^2 d
\Phi\big\rangle_g \tfrac{(F_+)^{2n}}{(2n)!} \\
=&\ \int_M \langle d \Phi, \theta_I \rangle_g \tfrac{(F_+)^{2n}}{(2n)!} -
\tfrac{1}{4}\int_M\big\langle (I-J) d \Phi, (I-J) d \Phi\big\rangle_g
\tfrac{(F_+)^{2n}}{(2n)!} .
\end{split}
\end{equation}
A similar calculation yields
 \begin{equation}\label{2-3}
  \int_M \gD_g \Phi \tfrac{(F_{-})^{2n}}{(2n)!}  = \int_M \langle d \Phi,
\theta_I \rangle_g \tfrac{(F_-)^{2n}}{(2n)!} + \tfrac{1}{4}\int_M\big\langle
(I+J)
d \Phi, (I+J) d \Phi \big\rangle_g \tfrac{(F_-)^{2n}}{(2n)!}.
\end{equation}
The claim follows by substituting (\ref{key-7}) in \eqref{2-1}, and using
\eqref{2-2} and \eqref{2-3} together with the basic fact
$$g ((I-J) X, (I-J) Y)+ g((I+J) X, (I+J) Y) = 4g(X, Y).$$
\end{proof}

\begin{proof}[Proof of Theorem \ref{flowandGIT}] The claims follow directly from
Propositions \ref{p:Hamiltonian-flow}, \ref{p:CY}, and \ref{p:energy}.
\end{proof}

\subsection{Conjectural picture}

Given the overall picture we have now shown, we make a natural conjecture
concerning the GKRF, adjoining Conjecture \ref{GCY-conj}.

\begin{conj} \label{conj:GKRF} Let $(M^{4n}, g, I, J)$ be a nondegenerate
generalized K\"ahler structure.  Then the solution to generalized K\"ahler Ricci
flow
exists for all time and converges to a hyper-K\"ahler metric.
\end{conj}

Reiterating the introduction, this conjecture is a natural analogue of Cao's
theorem \cite{Cao} establishing global existence and convergence of
K\"ahler-Ricci flow when $c_1 = 0$.  In \cite{SNDG} the second author
established the existence portion of this
conjecture in the case $n=1$, as well as a form of convergence to a kind of weak
hyper-K\"ahler structure.  We next show a result which indicates the main
analytic hurdle left to overcome to establish this conjecture.

\begin{thm} \label{formalflowtheorem2} Let $(M^{4n}, g_0, I, J)$ be a
nondegenerate
generalized K\"ahler structure and $(g_t, I, J_t)$ denote the solution to GKRF
with this initial condition in the $I$-fixed gauge.  Suppose there exists a
constant $\gL > 0$ such that for all times $t$ in the maximal interval of
existence, the solution satisfies
\begin{align*}
\gL^{-1} g_0 \leq g_t \leq \gL g_0. 
\end{align*}
Then the solution exists for all time and converges to a hyper-K\"ahler metric
$(g_{\infty}, I, J_{\infty})$. Furthermore $J_{\infty} = \phi_{\infty}^* J_0$
for some $\phi_{\infty} \in {\rm Ham}(M, \Omega)$.
\begin{proof} To show the long time existence we exploit regularity results from
\cite{SBIPCF}, assuming some technical familiarity with that paper.  

In our case $\omega = (F)^{1,1}_I$ where $F= (F_+)_t$ are the symplectic $2$-forms evolving in the same deRham class by the equation \eqref{omega-pm}. It follows that $\omega_t$ evolves in the same Aeppli class, i.e. we can write
\begin{equation}\label{reduced1}
\omega_t = \omega_0 + \bar \partial \alpha_t + \partial \bar \alpha,
\end{equation}
for a family of $(1,0)$-forms $\alpha_t \in \wedge^{1,0}(M, I)$. However,  $\alpha_t$ is not uniquely determined and in order to  make contact with the normalization used in \cite{SBIPCF}, we notice that,  as our manifold $(M^{4n}, I)$ admits a holomorphic volume form
$(\Omega_I^{2,0})^{n}$, there exists some background Hermitian metric $h$
such that $\rho_C(h) = 0$. Indeed, $h$ can be obtained by conformal modification of any given
Hermitian metric on $(M, I)$.  We can  thus apply the setting of (\cite{SBIPCF}, Sec. 3.1) with the choices $\hat{\gw}_t = (\gw_I)_0$ and $\mu = 0$ (noting that the normalization for the flow used in this paper differs by a factor  $2$ from the one used in \cite{SBIPCF}),  and  write \eqref{reduced1} with  $\alpha=\alpha_t$ being 
 a maximal solution to the $(\hat{g}_t, h, \mu)$-reduced pluriclosed flow
 \begin{equation}\label{reduced}
 \begin{split}
\frac{\partial}{\partial t} \alpha &= \bar \partial^*_{g_t} \omega_t - \frac{\sqrt{-1}}{2}\partial \log \frac{\det g_t}{\det h}, \\
\alpha (0) &= 0.
\end{split}
\end{equation}
With this choice of $\alpha$ the so called {\it torsion potential} $\partial \alpha$  evolves by
\begin{equation}\label{torsion-potential-1}
\frac{\partial}{\partial t} \partial  \alpha = \partial {\bar \partial}^*_{g_t} \omega_t =  \sqrt{-1}\partial (\theta_I^{1,0}),
\end{equation}
where $\theta_I = I \delta^{g_t}\omega_t$ is the Lee form of $(g_t, I)$. 
Using  \eqref{omega-pm}, Proposition~\ref{ricci-potential} and the relation \eqref{ricci-relation} (which shows that $(\rho_B)^{2,0}_I =-\sqrt{-1} \partial (\theta_I^{1,0})$), 
we deduce 
\begin{equation}\label{torsion-potential-2}
\frac{\partial}{\partial t} (F_t)^{2,0}_I  = \sqrt{-1}\partial (\theta_I^{1,0}),
\end{equation}
showing that the torsion potential satisfies
\begin{equation}\label{torsion-potential-3}
\partial \alpha_t = (F_t)_I^{2,0} - (F_0)^{2,0}_I.
\end{equation}
By  Lemma~\ref{l:non-degenerate-symplectic} and  \eqref{l:symplectic-unique},  we have $(F_t)^{2,0}_I=-\sqrt{-1} (b_t)^{2,0}_I$ with $b_t=-g_t(I+J_t)^{-1}(I-J_t)= -\Omega(I- J_t)^2.$  By Corollary~\ref{c:potential-bound}, $b^{2n}/\omega^{2n}=e^{\Phi}$  is uniformly bounded.  The a priori metric bound yields a universal bound of the volume form $\omega_t^{2n}$ with respect a background volume form. As $\Omega$ is non-degenerate and constant under the flow, we can use $\Omega^{2n}$ as a background volume, so we conclude $\det(I-J_t)^2$ is uniformly bounded;  using that  $(I-J_t)^2$ is positive definite and symmetric with respect to $g_t$,  and  the identity  ${\rm tr}_{g_t} (I-J_t)^2=8n(1- p_t), $ where $p_t= -\frac{1}{4n} {\rm tr}_{g_t} I J_t = \frac{1}{4n}\langle I, J_t\rangle_{g_t}$ is the angle function with $|p_t|<1$, we conclude that $||(I- J)^2||_{g_t}$, and whence $||b_t||^2_{g_t}$ and $||\partial \alpha_t||^2_{g_t}$,  are uniformly bounded. Summarizing, we have shown that the a priori metric bound yields  a torsion potential bound, i.e. we have
\begin{equation}\label{torison-potential-4}
\gL^{-1} g_0 \leq g_t \leq \gL g_0, \qquad \brs{\del \ga}_{g_0}^2 \leq \gL,
\end{equation}
 along the flow.

Using now the a priori
metric bound and torsion potential bound, we can apply
(\cite{SBIPCF} Theorem 1.7) to conclude uniform higher order regularity of the
time varying metric.  The 
global existence then  follows from standard arguments.

Given the smooth global existence and uniform estimates for the metric now in
place, we may choose any sequence of times $\{t_i\} \to \infty$ and obtain a
subsequential limit $\{(g_{t_i},J_{t_i})\} \to (g_{\infty}, J_{\infty})$, with
convergence in $C^{\infty}$.  Proposition \ref{gradspinordecay} implies that
this
limiting metric must satisfy $\Phi_{\infty} \equiv \gl$.  Proposition
\ref{ricci-potential} then implies that the limiting metric is $I$-Bismut Ricci
flat, hence by Proposition \ref{p:CY} $(g_{\infty},I)$ and
$(g_{\infty},J_{\infty})$ belong to the same hyper-K\"ahler structure.  Given
this subsequential convergence and moreover the existence of the hyper-K\"ahler
structure, we apply (\cite{ST1} Theorem 1.2) to conclude that the entire flow
converges exponentially fast to this same hyper-K\"ahler structure.  With this
exponential convergence in place, it follows that the time dependent vector
field $\N \Phi \hook \Omega$ driving the family of diffeomorphisms $\phi_t$ of
the $J$ evolution  as well as the normalized Hamiltonian function $\tilde
\Phi_t:= \Phi_t - \lambda$ converges exponentially fast to zero.  Letting $t :=
\tan(s),  s \in [0, \tfrac{\pi}{2})$, we obtain that  the reparametrized 
isotopy ${\tilde \phi}_{s}:= \phi_{\tan (s)}$  is generated by the $\Omega$
hamiltonian function $H_s:= \tfrac{1}{\cos^2 s} {\tilde  \Phi}_{\tan (s)}$
which, because of the exponential rate of convergence of $\tilde \Phi$,  can be
extended to a smooth function  defined on $[0, \tfrac{\pi}{2}]\times M$ by
letting $H_{\tfrac{\pi}{2}} \equiv 0$.  It thus follows that $\phi_t$ itself
converges to a limiting $\Omega$ Hamiltonian diffeomorphism $\phi_{\infty}$ such
that $\phi_{\infty}^* J_0 = J_{\infty}$. \end{proof}
\end{thm}

\section{Global existence and weak convergence on hyper-K\"ahler manifolds}
\label{ltesec}

In this section we prove Theorem \ref{flowtheorem}.  We break the proof into
three phases.  First we use the background K\"ahler structure to set up a
simplified reduction of the pluriclosed flow system.  Next we exploit this
special reduction and the evolution equations of \S \ref{flowsec} to obtain
global existence of the flow.  Finally we establish the weak convergence at
infinity.

\subsection{Reduction of pluriclosed flow} \label{s:reducedflow}

In this section we reduce the pluriclosed flow to a certain system coupling an
evolution for a $(1,0)$-form with a scalar evolution.  We exploit the
hyper-K\"ahler background to simplify the background terms needed to define this
reduced equation.  To begin we exhibit a proposition allowing us to express an
arbitrary pluriclosed metric as a K\"ahler metric plus a term coming from a
potential $(1,0)$-form.

\begin{prop}\label{pluriclosed-to-kahler} Let $g$ be a pluriclosed Hermitian
metric on a compact complex manifold $(M, I)$. If $(M, I)$ admits a K\"ahler
metric, then the Aeppli class of $\gw_I$ contains a K\"ahler metric $\gw$, i.e.
\begin{align*}
\gw_I = \omega + \partial \bga+ \bar{\partial} \ga
\end{align*}
for some $\alpha \in \wedge^{1,0}(M,I)$ and a positive definite closed $(1,1)$
form
$\omega$.
\end{prop}
\begin{proof} By Lemma~\ref{pluriclosed-to-symplectic}, $\gw_I$ is tamed by a
symplectic $2$-form $\psi$. Let $g'$ be any K\"ahler metric on $(M, I)$. By the
Hodge theorem, the $g'$-harmonic part $\psi_H$ of the closed $2$-form $\psi$
decomposes as the sum $(\psi_H)^{1,1} + (\psi_H)^{(2,0)+(0,2)}$ of harmonic
$2$-forms of type $(1,1)$ and $(2,0) + (0,2)$, respectively. Letting $\psi':=
\psi - (\psi_H)^{(2,0)+(0,2)}$ we obtain a new closed $2$-form with
$(\psi')^{1,1}= \psi^{1,1}= \gw_I$ and whose $g'$-harmonic part is of type
$(1,1)$.  We write $\psi$ instead of $\psi'$ and  we thus have shown that
$\gw_I$ is the $(1,1)$ part of a closed $2$-form $\psi$ which determines a
deRham class $[\psi] \in H^{1,1}(M, I)$. We want to prove that this class
contains a K\"ahler form $\gw$, since then writing $\psi = \gw + d\alpha$ and
taking $(1,1)$ part concludes the proof.

To this end, we use the deep result of Demailly-Paun (\cite{DP} Theorem 4.2),
which states that on a compact K\"ahler manifold $(M, I, \omega_{\mbox{\tiny
Kah}})$, a deRham class $\ga \in
H^{1,1}(M,I)$ is a K\"ahler class if and only if for every irreducible analytic
set $A\subset M$,  ${\rm dim}_{\mathbb C} A=p$, and every $t\ge 0$, one has
\begin{align} \label{ptk10}
\int_A (\alpha + t[\omega_{\mbox{\tiny Kah}}])^p >0.
\end{align}
In our situation, $\alpha = [\psi]$ is represented by an $I$-taming symplectic
$2$-form, and therefore so is any class $\alpha + t[\omega_{\mbox{\tiny Kah}}]$
for $t>0$.
Indeed,  the $(1,1)$ part of $\psi + t \omega_{\mbox{\tiny Kah}}$ is $\gw_I + t
\omega_{\mbox{\tiny Kah}} >0$.  It thus follows that $(\psi + t
\omega_{\mbox{\tiny Kah}})^p$ defines a strictly positive measure on the regular
part $A_{\rm reg}$ of $A$, thus showing that (\ref{ptk10}) holds.  Note that we
have used Lelong's theorem to define the integral $\int_A (\psi+ t
\omega_{\mbox{\tiny Kah}})^p$, see
e.g. (\cite{chirka} Chapter 4), and more specifically to say that the integral 
of $ (\psi+ t \omega_{\mbox{\tiny Kah}})^p$ over the regular part of $A$ exists
and is independent of the
choice of representative of $\ga + t[\omega_{\mbox{\tiny Kah}}]$.
\end{proof}

We now describe our reduction of the generalized K\"ahler-Ricci flow in the
$I$-fixed gauge.  Fix $(M^{4n}, I)$ a K\"ahler, holomorphic symplectic 
manifold, and suppose
$(M^{4n}, g, I, J)$ is some nondegenerate generalized K\"ahler structure.  By
Proposition~\ref{pluriclosed-to-kahler} we can choose $\ga \in
\wedge^{1,0}(M,I)$ such that
\begin{align*}
\gw_I = \gw_{\mbox{\tiny HK}} + \del \bga + \delb \ga,
\end{align*}
where $\gw_{\mbox{\tiny HK}}$ denotes the K\"ahler form of the  hyper-K\"ahler
metric on
$(M^{4n}, I)$ in the Aeppli class of $\omega_I$.  We will always normalize the
initial data for our flow in this
way without further comment.

\begin{lemma} \label{oneformreduction} Let $(M, I, g_t)$ be a
solution to pluriclosed flow, and suppose $\ga_t \in \wedge^{1,0}(M,I)$
satisfies
\begin{gather} \label{alphaflow}
\begin{split}
\dt \ga =&\ \delb^*_{\gw_t} \gw_t - \frac{{\i}}{2}\del \log \tfrac{\det
g_t}{\det
g_{\mbox{\tiny HK}}}\\
\ga(0) =&\ \ga_0,
\end{split}
\end{gather}
then the one-parameter family of pluriclosed metrics $\gw_{\ga} =
\gw_{\mbox{\tiny HK}}
+
\delb
\ga + \del \bga$ is the given solution to pluriclosed flow.
\begin{proof} This is a simple modification of (\cite{SBIPCF} Lemma 3.2).
\end{proof}
\end{lemma}

We note that the natural local decomposition of a pluriclosed metric as $\gw =
\gw_{\mbox{\tiny HK}}+
\del \bga +
\delb \ga$ is not canonical, as one may observe that $\ga + \del f$ describes
the same K\"ahler form for $f \in C^{\infty}(M, \mathbb R)$.  Because of this
``gauge-invariance,'' the equation (\ref{alphaflow}) is not parabolic, and admits
large families of equivalent solutions.  In \cite{SBIPCF} the second author
resolved this
ambiguity by giving a further reduced description of (\ref{alphaflow}) which is
parabolic.  In particular, as exhibited in (\cite{SBIPCF} Proposition 3.9) in
the case when the background metric is fixed and K\"ahler, if one has a family
of
functions $f_t$ and $(1,0)$-forms $\gb_t$ which satisfy
\begin{gather} \label{decflow}
\begin{split}
\dt \gb =&\ \gD^C_{g_t} \gb - T_{g_t} \circ \delb \gb\\
\dt f =&\ \gD^C_{g_t} f + \tr_{g_t} g_{\mbox{\tiny HK}} + \log \tfrac{\det
g_t}{\det g_{\mbox{\tiny HK}}}\\
\ga_0 =&\ \gb_0 - \i \del f_0,
\end{split}
\end{gather}
then $\ga_t := \gb_t - \i \del f_t$ is a solution to (\ref{alphaflow}).  The
term $T \circ \delb \gb$ is defined by
\begin{align} \label{Tterm}
(T \circ \delb \gb)_i = g^{\bl k} g^{\bq p} T_{ik\bq} \N_{\bl} \gb_p.
\end{align}

\subsection{Evolution equations}

In this subsection we record several evolution
equations and a priori
estimates directly associated to a solution to nondegenerate generalized
K\"ahler-Ricci flow.
We will in places refer to a solution to (\ref{decflow}), assuming the setup of
\S \ref{s:reducedflow}.

\begin{lemma} \label{dtfevlemma} Given a
solution to (\ref{decflow})
as above, one has
\begin{align*}
 \left( \dt - \gD^C_{g_t} \right) \tfrac{\del f}{\del t} =&\ \IP{\tfrac{\del
g}{\del t}, \delb \gb + \del \bgb}.
\end{align*}
\begin{proof}  This follows directly from the calculation of (\cite{SNDG}
Proposition 4.10), which we record here for convenience.
 \begin{align*}
 \dt \tfrac{\del f}{\del t} =&\ \tfrac{\del}{\del t} \left[ n - \tr_{g_t} \left(
\delb \gb + \del \bar{\gb}
\right) + \log \tfrac{\det g_t}{\det g_{\mbox{\tiny HK}}} \right]\\
 =&\ \IP{\tfrac{\del g}{\del t}, \delb \gb + \del \bar{\gb}} - \tr_{g_t} \left[
\dt \left( \delb \gb + \del \bar{\gb} \right) \right] + \tr_{g_t} \tfrac{\del
g}{\del t}\\
 =&\ \IP{\tfrac{\del g}{\del t}, \delb \gb + \del \bar{\gb}} + \tr_{g_t}
\del\delb \tfrac{\del f}{\del t}\\
 =&\ \gD_{g_t}^C f_t + \IP{\tfrac{\del g}{\del t}, \delb \gb + \del \bar{\gb}},
\end{align*}
as required.
\end{proof}
\end{lemma}

\begin{lemma} \label{betaevlemma} Given a
solution to (\ref{decflow})
as above, one has
\begin{align} \label{betanormev}
\left( \dt - \gD^C_{g_t} \right) \brs{\gb}^2 =&\ - \brs{\N \gb}^2 -
\brs{\bar{\N} \gb}^2 -
\IP{Q, \gb \otimes \bar{\gb}} + 2 \Re \IP{\gb, T \circ \delb \gb},
\end{align}
where
\begin{align} \label{Qdef}
Q_{i \bj} = g^{\bl k} g^{\bq p} T_{i k \bq} T_{\bj \bl p}.
\end{align}
\begin{proof} This is a simple modification of (\cite{SBIPCF} Proposition 4.4).
\end{proof}
\end{lemma}

\begin{cor} \label{betaestcor} Given a solution to
(\ref{decflow}) as above one has
\begin{align} \label{bnormev2}
\left( \dt - \gD^C_{g_t} \right) \brs{\gb}^2 \leq&\ - \brs{\N \gb}^2.
\end{align}
In particular, one has
\begin{align} \label{best}
\sup_M \brs{\gb_t}^2_{g_t} \leq \sup_M \brs{\gb_0}^2_{g_0}.
\end{align}
\begin{proof} A similar inequality and estimate was claimed in (\cite{SBIPCF}
Corollary 4.5) in the case $n=1$.  The central point is to obtain an inequality
bounding the final inner product term in (\ref{betanormev}).  In general we see
using (\ref{Tterm}) that in fact
\begin{align*}
2 \Re \IP{\gb, T \circ \delb \gb} =&\ 2 \Re \left( g^{\bj i} \gb_{\bj} \left(
g^{\bl k} g^{\bq p} T_{i k \bq} \N_{\bl} \gb_{p} \right) \right)\\
=&\ 2 \Re \left[ g^{\bl k} g^{\bq p} \left( g^{\bj i} \gb_{\bj} T_{i k \bq}
\right) \left( \N_{\bl} \gb_p \right) \right]\\
=&\ 2 \Re \IP{\gb^{\sharp} \hook T, \bar{\N} \gb}.
\end{align*}
Also we note using (\ref{Qdef}) that
\begin{align*}
\IP{Q, \gb \otimes \bar{\gb}} =&\ g^{\bk i} g^{\bj l} Q_{i \bj} \gb_{\bk}
\gb_l\\
=&\ g^{\bk i} g^{\bj l} \left( g^{\bq p} g^{\bs r} T_{i p \bs} T_{\bj \bq r}
\right) \gb_{\bk} \gb_l\\
=&\ g^{\bq p} g^{\bs r} \left( g^{\bk i} \gb_{\bk} T_{i p \bs} \right) \left(
g^{\bj l} \gb_l T_{\bj \bq r} \right)\\
=&\ \brs{\gb^{\sharp} \hook T}^2.
\end{align*}
Using these calculations we see by the Cauchy-Schwarz inequality that
\begin{align*}
\dt \brs{\gb}^2 =&\ \gD^C \brs{\gb}^2 - \brs{\N \gb}^2 - \brs{\bar{\N} \gb}^2 -
\IP{Q, \gb \otimes \bar{\gb}} + 2 \Re \IP{\gb, T \circ \delb \gb}\\
=&\ \gD^C \brs{\gb}^2 - \brs{\N \gb}^2 - \brs{\bar{\N} \gb}^2 -
\brs{\gb^{\sharp} \hook T}^2 + 2 \Re \IP{\bar{\N} \gb, \gb^{\sharp} \hook T}\\
\leq&\ \gD^C \brs{\gb}^2 - \brs{\N \gb}^2,
\end{align*}
as required.  The
estimate (\ref{best}) now follows directly from the maximum principle.
\end{proof}
\end{cor}

\begin{lemma} \label{delalphaevlemma} Given a solution to (\ref{alphaflow}) as
above, one has
\begin{align} \label{delalphaev}
\left(\dt - \gD_{g_t}^C \right) \brs{\del \ga}^2 =&\ - \brs{\N \del \ga}^2 -
\brs{T_{g_t}}^2 - 2
\IP{Q, \del \ga \otimes \delb \bga}.
\end{align}
In particular,
\begin{align} \label{delalphaest}
\sup_{M} \brs{\del \ga_t}^2_{g_t} \leq \sup_{M} \brs{\del \ga_0}^2_{g_0}.
\end{align}
\begin{proof} It follows from the proof of (\cite{SBIPCF} Proposition 4.9), with
$\hat{g} = g_{\mbox{\tiny HK}}$ a K\"ahler metric, and $\mu = 0$, that
\begin{align*}
\dt \del \ga =&\ \gD^C_{g_{\ga}} \del \ga - \tr_{g_{\ga}} \N^{g_{\ga}}
T_{\hat{g}}= \gD^C_{g_{\ga}} \del \ga.
\end{align*}
Equation (\ref{delalphaev}) now follows from (\cite{SPCFSTB} Lemma 4.7).  The
estimate
(\ref{delalphaest}) follows directly from the maximum principle.
\end{proof}
\end{lemma}

Next we record a few basic evolution equations associated to solutions of
pluriclosed flow.

\begin{lemma} \label{volumeformev} Let $(M^{2m}, I, g_t)$ be a solution to
pluriclosed flow, and suppose $h$ is another Hermitian metric on $(M,I)$.  Then
\begin{align*}
\left(\dt - \gD_{g_t}^C \right) \log \tfrac{\det g_t}{\det h} = \brs{T}^2 -
\tr_g
\rho_C(h).
\end{align*}
\end{lemma}

\begin{lemma} \label{traceev} Let $(M^{2n}, I, g_t)$ be a
solution to
pluriclosed flow, and let $h$ denote another Hermitian metric on $(M, I)$.  Then
\begin{align*}
\left(\dt - \gD^C_{g_t} \right) \tr_h g =&\ - \brs{\gU(g,h)}^2_{g^{-1},h^{-1},g}
+ \tr_h Q - g^{\bq p}
(\Omega^h)_{p \bq}^{\bl k} g_{k \bl},
\end{align*}
where $\gU(g,h) = \N^C_g - \N^C_h$ is the difference of the Chern connections
associated to $g$ and $h$.
\begin{proof} To begin we establish a general evolution equation for pluriclosed
flow which is implicit in (\cite{ST3} Proposition 2.4).  Using the formula for
pluriclosed flow in complex coordinates (cf. \cite{ST2} (1.3)) we obtain
\begin{align*}
 \dt \tr_h g =&\ \dt h^{\bj i} g_{i \bj} = h^{\bj i} \left[ g^{\bq p} g_{i
\bj,p\bq} - g^{\bq p}
g^{\bs r} g_{i \bs,p} g_{r \bj,\bq} + Q_{i \bj} \right].
\end{align*}
On the other hand
\begin{align*}
 \gD^C \tr_h g =&\ g^{\bq p} \left[ h^{\bj i} g_{i \bj} \right]_{,p\bq}\\
 =&\ g^{\bq p} \left[ - h^{\bj k} h_{k \bl,p} h^{\bl i} g_{i \bj} + h^{\bj i}
g_{i\bj,p} \right]_{,\bq}\\
 =&\ g^{\bq p} \left[ h^{\bj r} h_{r \bs,\bq} h^{\bs k} h_{k \bl,p} h^{\bl i}
g_{i \bj} - h^{\bj k} h_{k \bl,p\bq} h^{\bl i} g_{i \bj} + h^{\bj k} h_{k \bl,p}
h^{\bl r} h_{r\bs,\bq} h^{\bs i} g_{i \bj} \right.\\
&\ \qquad \left. - h^{\bj k} h_{k \bl,p} h^{\bl i} g_{i \bj,\bq} - h^{\bj k}
h_{k \bl,\bq} h^{\bl i} g_{i\bj,p} + h^{\bj i} g_{i \bj,p\bq} \right].
 \end{align*}
Combining the above calculations yields
\begin{align*}
 \left( \dt - \gD^C \right) \tr_h g =&\ - h^{\bj i} g^{\bq p}
g^{\bs r} g_{i \bs,p} g_{r \bj,\bq} + \tr_h Q\\
&\ - g^{\bq p} \left[ h^{\bj r} h_{r \bs,\bq} h^{\bs k} h_{k \bl,p} h^{\bl i}
g_{i \bj} - h^{\bj k} h_{k \bl,p\bq} h^{\bl i} g_{i \bj} + h^{\bj k} h_{k \bl,p}
h^{\bl r} h_{r\bs,\bq} h^{\bs i} g_{i \bj} \right.\\
&\ \qquad \left. - h^{\bj k} h_{k \bl,p} h^{\bl i} g_{i \bj,\bq} - h^{\bj k}
h_{k \bl,\bq} h^{\bl i} g_{i\bj,p} \right]\\
=&\ - \brs{\gU(g,h)}^2_{g^{-1},h^{-1},g} + \tr_h Q - g^{\bq p}
(\Omega^h)_{p \bq}^{\bl k} g_{k \bl},
\end{align*}
as required.
\end{proof}
\end{lemma}

\begin{lemma} \label{traceineq} Let $(M^{2m}, I, g_t)$ be a
solution to
pluriclosed flow, and let $h$ denote another Hermitian metric on $(M, I)$.  Then
there exists a constant $C$ depending on $h$ such that
\begin{align*}
\left( \dt - \gD^C_{g_t} \right) \log \tr_h g \leq&\ \brs{T}^2_g + C \tr_g h.
\end{align*}

\begin{proof} A direct calculation using Lemma \ref{traceev} implies that
\begin{align*}
 \left(\dt - \gD^C_{g_t} \right) \log \tr_h g =&\ \tfrac{1}{\tr_h g} \left[ -
\brs{\gU(g,h)}^2_{g^{-1},h^{-1},g} + \tr_h
Q - \tr_g h \tr_h g + n \tr_h g\right] + \tfrac{\brs{\N \tr_h g}^2}{(\tr_h
g)^2}.
\end{align*}
Since $h$ is K\"ahler we may choose complex normal coordinates at any given
point such
that
\begin{align*}
 h_{i \bj} = \gd_{ij}, \qquad g_{i \bj} = g_{i \bi} \gd_{ij}, \qquad \del_i h_{j
\bk} = 0.
\end{align*}
Then we may estimate using the Cauchy-Schwarz inequality
\begin{align*}
 \tfrac{\brs{\N \tr_h g}^2}{\tr_h g} =&\ \left( \sum_i g_{i \bi} \right)^{-1}
g^{\bj j} \N_j \tr_h g \N_{\bj} \tr_h g\\
 =&\ \left( \sum_i g_{i \bi} \right)^{-1}  \sum_j g^{\bj j} \left[ \sum_{k} \N_j
g_{k \bk} \sum_{l} \N_{\bj} g_{l \bl} \right]\\
 =&\ \left( \sum_i g_{i \bi} \right)^{-1}  \sum_j \left[ \sum_{k} \left[ (g^{\bj
j})^{\tfrac{1}{2}} \gU_{jk}^k (g_{k \bk})^{\tfrac{1}{2}} \right] (g_{k
\bk})^{\tfrac{1}{2}} \sum_{l} \left[ (g^{\bj j})^{\tfrac{1}{2}} \gU_{\bj
\bl}^{\bl} (g_{l \bl})^{\tfrac{1}{2}} \right] g_{l \bl} \right]\\
 \leq&\ \left( \sum_i g_{i \bi} \right)^{-1} \left( \sum_{j,k} g^{\bj j} g_{k
\bk} \gU_{j k}^k \gU_{\bj \bk}^{\bk} \right)^{\tfrac{1}{2}} \left( \sum_k g_{k
\bk} \right)^{\tfrac{1}{2}} \left( \sum_{j,l} g^{\bj j} g_{l \bl} \gU_{j l}^l
\gU_{\bj \bl}^{\bl} \right)^{\tfrac{1}{2}} \left( \sum_l g_{l \bl}
\right)^{\tfrac{1}{2}}\\
 \leq&\ \brs{\gU(g,h)}^2_{g^{-1},h^{-1},g}.
 \end{align*}
Moreover, note that in these same coordinates it follows that
\begin{align*}
(\tr_h g)^{-1} \tr_h Q =&\ \left( \sum_{i} g_{i \bar{i}} \right)^{-1} \sum_j
Q_{j \bar{j}} \leq \sum_i g_{i\bi}^{-1} Q_{i \bi} = \brs{T}_g^2.
\end{align*}
The result follows.
\end{proof}
\end{lemma}

\subsection{Global existence}

In this subsection we establish global
existence of the generalized K\"ahler-Ricci flow on a hyper-K\"ahler background.

\begin{prop} \label{lteprop} Let $(M^{4n}, g, I, J)$ be a nondegenerate
generalized K\"ahler manifold, and suppose $I$ is a K\"ahler complex structure. 
Then the solution to generalized K\"ahler Ricci flow with
initial condition $(g,I,J)$ exists on $[0,\infty)$.
\begin{proof} We use a solution $(\gb_t,f_t)$ to (\ref{decflow}) as above.  Our
main goal is to establish, for each finite time interval, uniform upper and
lower bounds for the metric tensor and the torsion potential estimate.  The
proposition will then follow along the
lines of the proof of Theorem \ref{formalflowtheorem2}.  To that end we first
apply
Lemma \ref{volumeformev} choosing $h = g_{\mbox{\tiny HK}}$, so that $\rho_C(h)
=
\rho_C(g_{\mbox{\tiny HK}}) = 0$, to obtain
\begin{align*}
\left(\dt - \gD_{g_t}^C \right) \log \tfrac{\det g_t}{\det g_{\mbox{\tiny HK}}}
=
\brs{T}^2
\geq 0.
\end{align*}
The maximum principle then implies
\begin{align} \label{lteprop10}
\inf_{M \times \{t\}} \log \tfrac{\det g_t}{\det g_{\mbox{\tiny HK}}} \geq&\
\inf_{M \times
\{0\}} \log \tfrac{\det g_t}{\det g_{\mbox{\tiny HK}}}.
\end{align}
Also, we can set
\begin{align*}
W_1 = \log \tfrac{\det g_t}{\det g_{\mbox{\tiny HK}}} + \brs{\del \ga}^2,
\end{align*}
and then combining Lemma \ref{delalphaevlemma} with Lemma \ref{volumeformev}
we obtain
\begin{align*}
\left(\dt - \gD_{g_t}^C \right) W_1 \leq&\ 0.
\end{align*}
The maximum principle then implies
\begin{align} \label{lteprop20}
\sup_{M \times \{t\}} \log \tfrac{\det g_t}{\det g_{\mbox{\tiny HK}}} \leq
\sup_{M \times
\{t\}} W_1 \leq \sup_{M \times \{0\}} W_1 \leq C.
\end{align}
Hence we have established uniform upper and lower bounds on the volume form.  To
finish the proof of uniform metric equivalence it suffices to show an upper
bound for the metric.

To that end we let
\begin{align*}
W_2 = \log \tr_{g_{\mbox{\tiny HK}}} g + \brs{\del \ga}^2 - A f,
\end{align*}
where $A$ is a constant to be determined.  Combining Lemmas
\ref{delalphaevlemma}, \ref{traceineq}, and recalling (\ref{decflow}) we see
\begin{align*}
\left(\dt - \gD_{g_t}^C \right) W_2 \leq&\ \left( C - A \right) \tr_{g}
g_{\mbox{\tiny HK}} -
A \log \tfrac{\det g_t}{\det g_{\mbox{\tiny HK}}}.
\end{align*}
Choosing $A$ sufficiently large with respect to $C$ and recalling the previously
established bound for the volume form yields
\begin{align*}
\left(\dt - \gD_{g_t}^C \right) W_2 \leq C.
\end{align*}
Applying the maximum principle we see that
\begin{align*}
\sup_{M \times \{t\}} \log \tr_{g_{\mbox{\tiny HK}}} g - A f \leq&\ \sup_{M
\times \{t\}} W_2
\leq \sup_{M \times \{0\}} W_2 + C t \leq C(1 + t).
\end{align*}
Rearranging yields
\begin{align} \label{lteprop30}
\sup_{M \times \{t\}} \tr_{g_{\mbox{\tiny HK}}} g \leq&\ \sup_{M \times \{t\}}
e^{C (1 + t +
f)}.
\end{align}
Hence, to finish the proof it suffices to estimate $f$.  Since we are only
concerned with finite time intervals, it suffices to estimate $\tfrac{\del
f}{\del t}$.

Thus set
\begin{align*}
 W_3 =&\ \tfrac{\del f}{\del t} + \brs{\gb}^2 - A_1
\log \tfrac{\det g}{\det g_{\mbox{\tiny HK}}} + A_2 \brs{\N \Phi}^2 ,
\end{align*}
where $A_1$ and $A_2$ are positive constants to be determined below. 
Combining Lemmas \ref{dtfevlemma}, \ref{betaevlemma}, and \ref{volumeformev}
with Proposition \ref{gradspinorev} (n.b. the conversion from Riemannian
Laplacian to Chern Laplacian when changing from $B$-field gauge to $I$-fixed
gauge), we obtain
\begin{gather*}
\begin{split}
 \left( \dt - \gD^C_{g_t} \right) W_3 =&\ \IP{\tfrac{\del g}{\del t}, \delb \gb
+
\del \bgb} + \left[ - \brs{\N \gb}^2 -
\brs{\bar{\N} \gb}^2 -
\IP{Q, \gb \otimes \bar{\gb}} + 2 \Re \IP{\gb, T \circ \delb \gb} \right]\\
&\ - A_1 \brs{T}^2 + A_2 \left[ -2 \brs{\N^2 \Phi}^2 - \tfrac{1}{2} \IP{\HH, \N
\Phi \otimes \N \Phi}\right].
\end{split}
\end{gather*}
First observe that by the Cauchy-Schwarz inequality and the a priori estimate
for $\gb$ we have
\begin{align*}
 2 \Re \IP{\gb, T \circ \delb \gb} \leq&\ C \brs{T} \brs{\bar{\N} \gb} \leq
\tfrac{1}{2} \brs{\bar{\N} \gb}^2 + C \brs{T}^2.
\end{align*}
Thus choosing $A_1$ sufficiently large and applying the Cauchy-Schwarz
inequality to $\IP{\tfrac{\del g}{\del t}, \delb \gb + \del \bgb}$, and dropping
negative terms we obtain
\begin{gather} \label{lteprop40}
 \begin{split}
 \left( \dt - \gD^C_{g_t} \right) W_3 \leq&\ \brs{\tfrac{\del g}{\del t}}^2 -
\tfrac{A_1}{2} \brs{T}^2 - 2 A_2 \brs{\N^2 \Phi}^2.
\end{split}
\end{gather}
Now note from Proposition \ref{ricci-potential} that $\tfrac{\del g}{\del t}$
can
be expressed as the $(1,1)$ projection of the $J$-Chern Hessian of the Ricci
potential $\Phi$.  Combining this with (\ref{Chernformula}), there is a uniform
constant $C$ such that
\begin{align} \label{lteprop50}
 \brs{\tfrac{\del g}{\del t}}^2 \leq&\ C \left[ \brs{\N^2 \Phi}^2 + \brs{T}^2
\brs{\N \Phi}^2 \right].
\end{align}
Since $\brs{\N \Phi}^2$ is uniformly bounded by Proposition
\ref{gradspinordecay},
plugging (\ref{lteprop50}) into (\ref{lteprop40}) and choosing $A_1$ and $A_2$
sufficiently large with respect to the initial data, we have
\begin{gather}
 \left(\dt - \gD_{g_t}^C \right) W_3 \leq 0.
\end{gather}
The a priori estimate for $W_3$ follows by the maximum principle.  Since the
volume form is bounded below uniformly, this implies an upper bound for
$\tfrac{\del f}{\del t}$ as required.  A directly analogous estimate can yield a
lower bound for $\tfrac{\del f}{\del t}$, finishing the proof.
\end{proof}
\end{prop}

\subsection{Weak convergence}

In this subsection we finish the proof of Theorem \ref{flowtheorem}.  First we
establish the convergence of the Ricci potential in $H_1^2$.  Then we derive specialized estimates to get the convergence to a closed current in the limit.

\begin{prop} \label{p:Phidecay} Let $(M^{4n}, g, I, J)$ be a nondegenerate
generalized K\"ahler manifold, and suppose that $I$ is a K\"ahler complex structure. If the solution to generalized K\"ahler
Ricci flow with
initial condition $(g,I,J)$ exists on $[0,\infty)$, then   there exists a constant $C$
such that
\begin{align*}
 \brs{\brs{\Phi - \gl}}_{H_1^2}^2 \leq C t^{-1}.
\end{align*}
\begin{proof} By the Poincar\'e Lemma for some background  K\"ahler metric $(\til g, I, \til{\gw})$, it
suffices to obtain the estimate for $\brs{\brs{d \Phi}}_{L^2_{\til g}}^2$. 
We proceed to estimate, using
properties of exterior algebra and Proposition~\ref{gradspinordecay},
\begin{align*}
 \brs{\brs{d \Phi}}_{L^2({\til g})}^2 =&\ \int_M \i \del \Phi \wedge \delb \Phi \wedge
\til{\gw}^{(m-1)}\\
 \leq &\ C\int_M \brs{d \Phi}_{\gw_t}^2 \gw_t \wedge {\tilde \omega}^{(m-1)}\\
 \leq&\ C t^{-1} \int_M \gw_t \wedge {\tilde \omega}^{(m-1)}\\
 =&\ C t^{-1} \int_M \left( \gw_{\tiny {\rm HK}} + \del \bga_t + \delb \ga_t \right) \wedge
{\tilde \omega}^{(m-1)}\\
 =&\ C t^{-1},
 \end{align*}
where the last line follows by Stokes Theorem and the fact that $\tilde \omega$ is  a closed $(1,1)$-form, and we recall that $m=2n$ is the complex dimension of $(M,I)$.
\end{proof}
\end{prop}

\begin{prop} \label{torsiondecayprop} Let $(M^{4n}, g, I, J)$ be a nondegenerate
generalized K\"ahler manifold, and suppose $I$ is a K\"ahler complex structure. 
Let $(g_t, I, J_t)$ be the solution to generalized K\"ahler Ricci flow in the
$I$-fixed gauge.  Choose $A > 0$ so that
\begin{align*}
 1 \leq W := \left( A - \log \frac{\det g}{\det g_{\mbox{\tiny{\rm HK}}}} \right)
\leq C,
\end{align*}
which exists by the estimates (\ref{lteprop10}) and (\ref{lteprop20}) of
Proposition \ref{lteprop}.  For $p > 1$ sufficiently large, one has
\begin{align*}
 \frac{d}{dt} \int_M W^p dV_g \leq&\ - \int_M \brs{T}^2 dV_g + C t^{-1}.
\end{align*}
\begin{proof} An elementary calculation yields
\begin{align*}
 \left( \dt - \gD^C_{g_t} \right) f^p =&\ p f^{p-1} \left(\dt - \gD^C_{g_t} \right) f -
p(p-1) f^{p-2} \brs{\N f}^2.
\end{align*}
Note also that (\ref{bogdan-stefan-2}) implies the general integration identity
\begin{align*}
 \int_M \gD^C_{g_t} f dV_g =&\ \int_M f \left( \brs{\theta}^2 - \tfrac{1}{2} \brs{T}^2
\right) dV_g.
\end{align*}
Combining these facts and using Lemma \ref{volumeformev},
\begin{align*}
 \frac{d}{dt} \int_M W^p dV_g =&\ \int_M \left[ \left( \dt W^p \right) + W^p
\left( \tr_{\gw_I} d J d \Phi \right) \right]dV_g\\
 =&\ \int_M \left\{ \left[ \gD^C_{g_t} W^p + p W^{p-1} \left( - \brs{T}^2 \right) -
p(p-1) W^{p-2} \brs{\N W}^2 \right] + W^p \left( \tr_{\gw_I} d J d \Phi \right)
\right\}dV_g\\
 \leq&\ \int_M \left\{ W^p \brs{\theta}^2 - p W^{p-1} \brs{T}^2 - p(p-1) W^{p-2}
\brs{\N W}^2 + W^p \tr_{\gw_I} d J d \Phi \right\} dV_g\\
 \leq&\ \int_M \left\{ - \frac{p}{2 C} W^{p} \brs{T}^2 - p(p-1) W^{p-2} \brs{\N
W}^2 + W^p \tr_{\gw_I} d J d \Phi \right\} dV_g\\
 =&\ A_1 + A_2 + A_3.
 \end{align*}
Note that the second inequality follows by choosing $p$ large with respect to
the bounds on $W$.  It remains to estimate $A_3$.  To that end we have
 \begin{align*}
  A_3 =&\ \int_M W^p d J d \Phi \wedge \gw_I^{n-1}\\
  =&\ \int_M p W^{p-1} d W \wedge J d \Phi \wedge \gw_I^{n-1} + W^p J d \Phi
\wedge d \gw_I \wedge \gw_I^{n-2}\\
  \leq&\ \gd_1 \int_M p W^{p-2} \brs{\N W}^2 dV_g + C \gd_1^{-1} \int_M \brs{\N
\Phi}^2 dV_g + \gd_2 \int_M W^p \brs{T}^2 dV_g + C \gd_2^{-1} \int_M \brs{\N \Phi}^2
dV_g\\
  \leq&\ \tfrac{1}{2} A_1 + \tfrac{1}{2} A_2 + C t^{-1},
 \end{align*}
 where the last line follows by choosing $\gd_1$ and $\gd_2$ small with respect
to universal constants, then applying Proposition \ref{gradspinordecay}.
The proposition follows.
\end{proof}
\end{prop}

\begin{lemma} \label{L2lemma10} Let $(M^{2m}, \gw, J)$ be a Hermitian manifold,
with $\gw'$
another Hermitian metric.  Given $\mu \in \wedge^{m,m-2}(M, I)$ one has
\begin{align*}
 \brs{\brs{\mu}}_{L^2(g)}^2 \leq&\ \sup_M \brs{\mu}_{\gw'}^2 \left(\frac{\det
g}{\det g'} \right)^{-1} \int_M \gw \wedge \gw \wedge (\gw')^{m-2}.
\end{align*}
\begin{proof} Fix a point $p \in M$ and choose complex coordinates such that
\begin{align*}
 \gw'_{i \bj} = \gd_{ij}, \qquad \gw_{i \bj} =&\ \gl_{i} \gd_{ij}
\end{align*}
Then we observe that
\begin{align*}
 \brs{\mu}^2_{\gw} dV_g =&\ \mu_{i_1\dots i_m \bj_1 \dots \bj_{m-2}} \bmu_{\bi_1
\dots \bi_{m} j_1 \dots j_{m-2}} g^{\bi_1 i_1} \dots g^{\bi_m i_m} g^{\bj_1 j_1}
\dots g^{\bj_{m-2} j_{m-2}} \left( \prod_{k=1}^m g_{k \bk} dV_{g'} \right)\\
\leq&\ \brs{\mu}_{\gw'}^2 \sum_{1 \leq j_1 < \dots < j_{m-2} \leq m}
\gl_{j_1}^{-1} \dots \gl_{j_{m-2}}^{-1} dV_{g'}\\
=&\ \brs{\mu}_{\gw'}^2 \left( \frac{\det g}{\det g'} \right)^{-1}
\left(\frac{\det g}{\det g'} \right) \sum_{1 < j_1 \dots < j_{m-2} \leq m}
\gl_{j_1}^{-1} \dots \gl_{j_{m-2}}^{-1} dV_{g'}\\
=&\ \brs{\mu}_{\gw'}^2 \left( \frac{\det g}{\det g'} \right)^{-1} \sum_{1 \leq
j_1 < j_2 \leq m} \gl_{j_1} \gl_{j_1} dV_{g'}\\
=&\ \brs{\mu}_{\gw'}^2 \left( \frac{\det g}{\det g'} \right)^{-1} \gw \wedge \gw
\wedge (\gw')^{m-2}.
\end{align*}
Integrating yields the result. 
\end{proof}
\end{lemma}

\begin{prop} \label{closedlimitprop} Let $(M^{4n}, g, I, J)$ be a nondegenerate
generalized K\"ahler manifold, and suppose $I$ is a K\"ahler complex structure. 
Let $(g_t, I, J_t)$ be the solution to generalized K\"ahler Ricci flow in the
$I$-fixed gauge.  Suppose $\{t_j\} \to \infty$ is a sequence such that
\begin{align*}
\lim_{j \to \infty} (\gw_I)_{t_j} = \gw_I^{\infty}, \qquad \lim_{j \to \infty}
\int_M \brs{T}_{g_{t_j}}^2 dV_{g_{t_j}} = 0,
\end{align*}
where $\gw_I^{\infty}$ is a positive $(1,1)$ current and the convergence is in the topology of
currents.  Then $\gw_I^{\infty}$ is closed.
\begin{proof} We will denote $(g_t, (\gw_I)_{t})$ by $(g, \gw)$ in this proof for notational
simplicity.  We fix a form $\mu \in \wedge^{m-1,m-2}(M, I)$ and compute
\begin{align*}
\int_M \gw \wedge \delb \mu =&\ \int_M \delb \gw \wedge \mu\\
=&\ \int_M \delb \del \bgb \wedge \mu\\
=&\ \int_M \delb \bgb \wedge \del \mu\\
\leq&\ \brs{\brs{\delb \bgb}}_{L^2(g)} \brs{\brs{\del \mu}}_{L^2(g)}.
\end{align*}
Note first the estimate using Lemma \ref{L2lemma10}, (\ref{lteprop10}) and the fact that the deRham class of the corresponding symplectic form $F_+=(F_+)_t$ does not change along the flow
\begin{align*}
 \brs{\brs{\del \mu}}_{L^2(g)}^2 \leq&\ C \sup_{M} \left( \frac{\det g}{\det
g_{\mbox{\tiny{HK}}}} \right)^{-1} \int_M \gw \wedge \gw \wedge
\gw_{\mbox{\tiny{HK}}}^{m-2} \\
 \leq &  C \sup_{M} \left( \frac{\det g}{\det
g_{\mbox{\tiny{HK}}}} \right)^{-1} \int_M F_+ \wedge F_+\wedge
\gw_{\mbox{\tiny{HK}}}^{m-2}  \leq  C.
 \end{align*}
Also we estimate using Corollary \ref{betaestcor} and Lemma \ref{delalphaevlemma}.
\begin{align*}
\brs{\brs{\delb \bgb}}_{{L^2}(g)}^2 =&\ \int_M \delb \bgb \wedge \del \gb \wedge
\gw^{m-2}\\
=&\ \int_M \bgb \wedge \delb \del \gb \wedge \gw^{m-2} + (m-2) \int_M\gb \wedge
\del \gb \wedge \delb \gw \wedge \gw^{m-3}\\
\leq&\ \sup_M \brs{\gb} \int_M \brs{T} dV_g + C \sup_M \brs{\gb} \brs{\del \gb}
\int_M \brs{T}_g dV_g\\
\leq&\ C \left( \int_M \brs{T}^2 dV_g \right)^{\tfrac{1}{2}}
\Vol(g)^{\tfrac{1}{2}}\\
=&\ o(j^{-1}).
\end{align*}
Combining these estimates it follows that
\begin{align*}
\int_M \gw_I^{\infty} \wedge \delb \mu = 0,
\end{align*}
as required.
\end{proof}
\end{prop}

\begin{prop} \label{p:closedlimitexists} Let $(M^{4n}, g, I, J)$ be a
nondegenerate
generalized K\"ahler manifold, and suppose $I$ is a K\"ahler complex structure. 
Let $(g_t, I, J_t)$ be the solution to generalized K\"ahler Ricci flow in the
$I$-fixed gauge.  There exists a sequence $\{t_j\} \to \infty$ and a closed positive
current $\gw_I^{\infty}$ such that 
 \begin{align*}
  \lim_{j \to \infty} (\gw_I)_{t_j} = \gw_I^{\infty}.
 \end{align*}
\begin{proof} First, since the quantity $W$ of Proposition~\ref{torsiondecayprop} is positive, it follows that there must exist some
sequence $\{t_j\} \to \infty$ such that
\begin{align*}
 \lim_{j \to \infty} \int_M \brs{T}^2_{g_{t_j}} dV_{g_{t_j}} = 0.
\end{align*}
Indeed, if $\liminf_{t \to \infty} \int_M \brs{T}^2 dV_g = \gd > 0$, then for sufficiently large $t > 0$ Proposition \ref{torsiondecayprop} yields
\begin{align*}
\frac{d}{dt} \int_M W^p dV_g \leq&\ - \tfrac{\gd}{2},
\end{align*}
which eventually yields a negative value for $\int_M W^p dV_g$, a contradiction.
Furthermore, by Lemma~\ref{oneformreduction} we have $(\omega_I)_t=
\omega_{\mbox{\tiny HK}} + \partial \bar \alpha_t + \bar \partial \alpha_t$,
and therefore
$$\int_M (\omega_I)_t \wedge \omega_{\mbox{\tiny HK}}^{2n-1} = \int_M
\omega_{\mbox{\tiny HK}}^{2n}.$$
By Banach-Alaoglu Theorem (see \cite{demailly-book}, Chapter III Proposition
1.23), the sequence $(\omega_I)_{t_j}$ weakly subsequently converges to a
positive current $\omega_I^{\infty}$.  We have thus obtained a sequence of times
satisfying the hypotheses of Proposition~\ref{closedlimitprop}, and hence
$\gw_I^{\infty}$ is closed.
\end{proof}
\end{prop}

\begin{proof}[Proof of Theorem \ref{flowtheorem}] The claims of the theorem
follow from Propositions \ref{lteprop}, \ref{p:Phidecay} and
\ref{p:closedlimitexists}.
\end{proof}

\begin{proof}[Proof of Corollary \ref{c:toriconv}] Exponential convergence of pluriclosed flow on tori to a flat K\"ahler metric is established in \cite{SBIPCF} Theorem 1.1.  Thus the GKRF will also converge as claimed.  Since GKRF preserves the $\Omega$-Hamiltonian deformation class by Theorem \ref{flowandGIT}, Conjecture \ref{CYconj} follows in this case as a consequence.
\end{proof}

\bibliographystyle{hamsplain}

\end{document}